\documentclass[11pt]{amsart}
\usepackage[bb=fourier,cal=pxtx]{mathalfa}
\usepackage{amssymb,amsfonts,graphicx,subfig}
\usepackage{mathtools}		
\usepackage{mathabx}		  
\usepackage{indentfirst}	
\usepackage[normalem]{ulem}  
\usepackage{amsthm}
\usepackage{thmtools}
\usepackage{enumitem}		
\usepackage{multirow,bigdelim} 
\usepackage[backref=page,colorlinks=true]{hyperref}	

\usepackage[font=small]{caption}

\usepackage{xcolor}
\hypersetup{
    colorlinks,
    linkcolor={red!50!black},
    citecolor={blue!50!black},
    urlcolor={blue!80!black}
}

\renewcommand*{\backref}[1]{}
\renewcommand*{\backrefalt}[4]{\quad \tiny 
  \ifcase #1 (\textbf{NOT CITED.})%
  \or    (Cited on page~#2.)%
  \else   (Cited on pages~#2.)%
  \fi}

\declaretheoremstyle[
headfont=\small\itshape,
bodyfont=\small
]{myremark}

\declaretheorem{theorem}
\declaretheorem[numberwithin=section]{otherthm}
\declaretheorem[sibling=otherthm]{lemma}
\declaretheorem[sibling=otherthm]{corollary}
\declaretheorem[sibling=otherthm]{proposition}

\declaretheorem[sibling=otherthm,style=remark]{conjecture}

\declaretheorem[style=remark]{question}
\declaretheorem[name=Acknowledgements, style=remark, numbered=no]{ack}

\hypersetup{bookmarksdepth = 3} 
\numberwithin{equation}{section}     

\setlist[enumerate,1]{label={\upshape(\alph*)},ref=\alph*}
\setlist[enumerate,2]{label={\upshape(\arabic*)},ref=\arabic*}

\newcommand{\WEDGE}{\mathsf{\Lambda}}  
\newcommand{\iii}{\mathtt{i}}
\newcommand{\jjj}{\mathtt{j}}
\newcommand{\kkk}{\mathtt{k}}

\newcommand{\tribar}[1]{\mathopen{| {\kern -1.5pt} | {\kern -1.5pt} |} {#1} \mathclose{| {\kern -1.5pt} | {\kern -1.5pt} |}}
\newcommand{\bigtribar}[1]{\mathopen{\big|{\kern -1.5pt}\big|{\kern -1.5pt}\big|}{#1}\mathclose{\big|{\kern -1.5pt}\big|{\kern -1.5pt}\big|}}
\newcommand{\Bigtribar}[1]{\mathopen{\Big|{\kern -1.5pt}\Big|{\kern -1.5pt}\Big|}{#1}\mathclose{\Big|{\kern -1.5pt}\Big|{\kern -1.5pt}\Big|}}
\newcommand{\biggtribar}[1]{\mathopen{\bigg|{\kern -1.5pt}\bigg|{\kern -1.5pt}\bigg|}{#1}\mathclose{\bigg|{\kern -1.5pt}\bigg|{\kern -1.5pt}\bigg|}}
\newcommand{\biangle}[1]{\mathopen{\langle {\kern -1.5pt} \langle}{#1} \mathclose{\rangle {\kern -1.5pt} \rangle}}

\renewcommand{\epsilon}{\varepsilon}
\renewcommand{\setminus}{\smallsetminus}
\renewcommand{\emptyset}{\varnothing}

\begin{document}

\title[Equilibrium states of singular value potentials]{Equilibrium states of generalised singular value potentials and applications to affine iterated function systems}

\author[J.~Bochi]{Jairo Bochi$^1$}
\address{Facultad de Matem\'aticas\\ 
	Pontificia Universidad Cat\'olica de Chile\\
	Avenida Vicu\~na Mackenna 4860\\
	Santiago, Chile}
\email{\href{mailto:jairo.bochi@mat.uc.cl}{jairo.bochi@mat.uc.cl}}

\author[I.D.~Morris]{Ian D. Morris$^2$}
\address{Department of Mathematics\\
	University of Surrey\\
	Guildford GU2 7XH\\
	United Kingdom}
\email{\href{mailto:i.morris@surrey.ac.uk}{i.morris@surrey.ac.uk}}

\begin{thanks}
{{}$^1$ Partially supported by projects Fondecyt 1140202, 1180371, and Conicyt PIA ACT172001.
{}$^2$ Partially supported by the Leverhulme Trust (grant number RPG-2016-194).}
\end{thanks}

\begin{abstract}
We completely describe the equilibrium states of a class of potentials over the full shift which includes Falconer's singular value function for affine iterated function systems with invertible affinities. We show that the number of distinct ergodic equilibrium states of such a potential is bounded by a number depending only on the dimension, answering a question of A. K\"aenm\"aki. We prove that all such equilibrium states are fully supported and satisfy a Gibbs inequality with respect to a suitable subadditive potential.  We apply these results to demonstrate that the affinity dimension of an iterated function system with invertible affinities is always strictly reduced when any one of the maps is removed, resolving a folklore open problem in the dimension theory of self-affine fractals. We deduce a natural criterion under which the Hausdorff dimension of the attractor has the same strict reduction property.\end{abstract}
\date{\today}

\maketitle

\section{Introduction and context}
If $T_1,\ldots,T_N \colon \mathbb{R}^d\to\mathbb{R}^d$ are contractions it is well-known that there exists a unique nonempty compact set $X\subset \mathbb{R}^d$ which solves the equation $X=\bigcup_{i=1}^NT_iX$. In this article we will refer to the tuple $(T_1,\ldots,T_N)$ as an \emph{iterated function system} and the set $X$ as its \emph{attractor}. When the transformations $T_i$ are all similarity transformations the set $X$ is called \emph{self-similar}; when they are assumed only to be affine transformations $X$ is called \emph{self-affine}. 
Subject to suitable hypotheses which guarantee that the distinct images $T_iX$ do not substantially overlap, the dimension theory of self-similar sets has been well understood since the work of J.E. Hutchinson in 1981 \cite{Hu81}. (Investigation of the overlapping case remains an active and challenging research topic: see e.g. \cite{FrHeOlRo15,Ho14,Ho15,ShSo16}.) The dimension theory of self-affine sets, by contrast, has been a source of stubborn open problems since its initiation in the 1980s by Bedford \cite{Be84}, McMullen \cite{Mc84} and Falconer \cite{Fa88}, and affine iterated function systems remain the focus of substantial research interest (see e.g. \cite{Ba07,BaKa15,BaRa17,BaFe13,DaSi16,FeSh14,Fr16,JoPoSi07,Mo16,MoSh16,Ra17}). A persistent feature of the literature on affine iterated function systems 
has been the requirement for additional hypotheses on the linear parts of the affinities in order to obtain results: they may be required to be positive or dominated \cite{Ba15,BaRa17,FaKe15,FaKe16,HuLa95}, to be of low dimension \cite{Ba07,Ba15,BaRa17,Be84,FaKe16,Fr12,HuLa95,KaMo16,Mc84,MoSh16}, to have a simple algebraic structure \cite{Ba07,Be84,DaSi16,FaMi07,Fr12,Mc84} or to induce invariant measures on projective space which themselves satisfy suitable dimension hypotheses \cite{Ba15,FaKe16,MoSh16,Ra17}. In this article we contribute to the still very small literature of results on affine iterated function systems which require no hypotheses whatsoever on the affinities other than that they be contracting and invertible. 

We recall that an iterated function system $(T_1,\ldots,T_N)$ satisfies the \emph{open set condition} if there exists a nonempty open set $U \subseteq \mathbb{R}^d$ such that $T_iU\subseteq U$ for every $i=1,\ldots,N$ and $T_iU \cap T_jU=\emptyset$ when $i \neq j$. If $T_1,\ldots,T_N$ are similarities satisfying the open set condition with $T_i$ having contraction ratio $r_i \in (0,1)$, a well-known theorem of Hutchinson \cite{Hu81} asserts that the Hausdorff dimension $s$ of the attractor satisfies the equation $\sum_{i=1}^N r_i^s=1$. Let us note three trivial consequences of this formula: firstly, if the maps $T_i$ are perturbed within this class then the value of the Hausdorff dimension predicted by the formula varies continuously with the perturbation; secondly, the value of $s$ may clearly be computed to within any prescribed accuracy in a finite amount of time when the contraction ratios $r_i$ are known; thirdly, if one of the maps $T_i$ is deleted then the value of the dimension predicted by the formula is strictly decreased. The extent of the difficulties presented by affine iterated function systems may perhaps be appreciated by observing that in the affine context an analogue of the first property was not established until 2014 by D.-J. Feng and P. Shmerkin \cite{FeSh14} and an analogue of the second property was unknown until established by the second named author in the recent article \cite{Mo16}. Prior to the present article the third property was known in the affine context only in dimensions three and lower \cite{KaMo16} or when the affinity dimension (defined below) is a rational number \cite{KaLi17}. As a corollary of the main result of this article we will establish the third of these three properties unconditionally for invertible affine iterated function systems of arbitrary dimension.

Let us describe the appropriate generalisation of Hutchinson's formula to affine iterated function systems. When $T_i$ is a similarity transformation all of the essential information about the $s$-dimensional volume of the image of the unit ball is captured by its contraction ratio $r_i$, but when $T_i$ is an affine transformation more detailed information is required. If $A$ is a linear transformation of $\mathbb{R}^d$ we recall that the \emph{singular values} of $A$ are defined to be the non-negative square roots of the eigenvalues of the positive semidefinite linear map $A^\top A$. We will write the singular values as $\alpha_1(A) \geq \alpha_2(A) \geq \cdots \geq \alpha_d(A)$, allowing repetition in the case of multiple eigenvalues. The existence of the singular value decomposition of $A$ implies that the image of the unit ball under $A$ is an ellipsoid with the lengths of the semiaxes equal to the singular values of $A$. 
Given a real number $s>0$ and linear transformation $A$ of $\mathbb{R}^d$ we define the \emph{singular value function} $\varphi^s(A)$ by
\[\varphi^s(A)\coloneqq \left\{\begin{array}{cl}
\alpha_1(A)\cdots \alpha_{\lfloor s\rfloor}(A)\alpha_{\lceil s\rceil }(A)^{s-\lfloor s\rfloor} &0\leq s \leq d\\
|\det A|^{s/d}& s \geq d. \end{array}\right.\]
It is well-known that $\varphi^s(AB) \leq \varphi^s(A)\varphi^s(B)$ for all linear transformations $A,B$ of $\mathbb{R}^d$. If $(T_1,\ldots,T_N)$ is an iterated function system on $\mathbb{R}^d$ defined by $T_ix=A_ix+v_i$ for linear maps $A_i$ and vectors $v_i \in \mathbb{R}^d$ we define the \emph{pressure} $P(\varphi^s)$ of $(T_1,\ldots,T_n)$ to be the limit
\[P(\varphi^s)=P((A_1,\ldots,A_N),s):=\lim_{n \to \infty} \frac{1}{n}\log \sum_{i_1,\ldots,i_n=1}^N \varphi^s\left(A_{i_n}\cdots A_{i_1}\right),\]
a quantity introduced by Falconer in \cite{Fa88}. The existence of the limit is guaranteed by subadditivity. For fixed invertible contractions $(T_1,\ldots,T_N)$ the function $s \mapsto P(\varphi^s)$ is continuous and strictly decreasing and has a unique zero, which we call the  \emph{affinity dimension}  of $(T_1,\ldots,T_N)$ and denote by $\dim_{\mathrm{aff}}(T_1,\ldots,T_N)$.  When all of the transformations $T_i$ are similarities with respective contraction ratios $r_i$ the equation $P(\varphi^s)=0$ simplifies to Hutchinson's formula $\sum_{i=1}^Nr_i^s=1$.

By contrast to the case of self-similar sets, the problem of finding sufficient conditions for the Hausdorff dimension of a self-affine set to equal the affinity dimension of the defining iterated function system is one of notorious difficulty. It was shown by Falconer \cite{Fa88} that the affinity dimension is always an upper bound for the Hausdorff dimension of the attractor, but the problem of finding explicit general criteria for the affinity dimension to also be a lower bound for the Hausdorff dimension remains challenging. As is typically the case in dimension theory, attention has focused on the construction of measures on the attractor whose Hausdorff dimension approximates or equals the desired lower bound.

Let us describe a key mechanism by which such measures of maximal dimension might be found. Given $(T_1,\ldots,T_N)$ with $T_ix=A_ix+v_i$ let $X$ denote the attractor of $(T_1,\ldots,T_N)$, let $\Sigma_N\coloneqq \{1,\ldots,N\}^{\mathbb{N}}$ and let $\sigma \colon \Sigma_N \to \Sigma_N$ denote the shift transformation $\sigma[(x_m)_{m=1}^\infty]\coloneqq (x_{m+1})_{m=1}^\infty$. We equip $\Sigma_N$ with the infinite product topology, which respect to which $\Sigma_N$ is compact and metrisable and $\sigma$ is continuous. Let $\mathcal{M}_\sigma$ denote the set of all $\sigma$-invariant Borel probability measures on $\Sigma_N$.  Given any $(x_m)_{m=1}^\infty \in \Sigma_N$ and $v \in \mathbb{R}^d$ there exists a unique limit
\[\pi\left[(x_m)_{m=1}^\infty\right] \coloneqq  \lim_{n \to \infty} T_{x_1}T_{x_2}\cdots T_{x_n}v \in X\]
which is independent of the choice of $v \in \mathbb{R}^d$. If $\mu$ is a fixed $\sigma$-invariant Borel probability measure on $\Sigma_N$ then the function
\[s \mapsto h(\mu)+ \lim_{n \to \infty} \frac{1}{n}\log \int_{\Sigma_N} \varphi^s(A_{x_n}\cdots A_{x_1})d\mu[(x_m)_{m=1}^\infty]\]
is well-defined, continuous and strictly decreasing and has a unique zero which we call the \emph{Lyapunov dimension} of $\mu$, denoted $\dim_{\mathrm{Lyap}}\mu$. (Here $h(\mu)$ denotes the Kolmogorov-Sinai entropy of $\mu$ with respect to $\sigma$.)  One may show without difficulty that
\[\dim_H\pi_*\mu \leq \dim_{\mathrm{Lyap}}\mu\leq \dim_{\mathrm{aff}}(T_1,\ldots,T_N).\]
If therefore one wishes to contruct a measure on $X$ with Hausdorff dimension equal to $\dim_{\mathrm{aff}}(T_1,\ldots,T_N)$ by projecting a $\sigma$-invariant measure $\mu$ from $\Sigma_N$ onto $X$, one must necessarily choose the measure $\mu$ in such a way that its Lyapunov dimension is equal to the affinity dimension of $(T_1,\ldots,T_N)$. Such measures are precisely those elements of $\mathcal{M}_\sigma$ which maximise the quantity
\begin{equation}\label{eq:var}h(\mu)+ \lim_{n \to \infty} \frac{1}{n}\log \int_{\Sigma_N} \varphi^s(A_{x_n}\cdots A_{x_1})d\mu\left[(x_m)_{m=1}^\infty\right]\end{equation}
where $s\coloneqq \dim_{\mathrm{aff}}(T_1,\ldots,T_N)$. In order to construct measures on self-affine sets whose Hausdorff dimension realises the affinity dimension, then, it is desirable to be able to describe and characterise precisely those measures which maximise \eqref{eq:var}. We refer to such measures as \emph{equilibrium states of $(A_1,\ldots,A_N)$ with respect to $\varphi^s$}. When the transformations  $T_i$ are similarities there is a unique equilibrium state, it is a Bernoulli measure, and under the open set condition on the maps $T_i$ it projects to a measure with Hausdorff dimension equal to that of the attractor \cite{Hu81}. For general affine contractions the existence of at least one equilibrium state follows from abstract considerations involving the semicontinuity of the functional \eqref{eq:var} on $\mathcal{M}_\sigma$ (see \cite{Ka04b}) but the structure and properties of equilibrium states have largely remained elusive, especially in higher dimensions and when the matrices $A_i$ are not positive or dominated. 

\begin{figure}%
    \centering
    \subfloat[The classical Sierpi\'nski gasket $X_1$.]{{\includegraphics[width=5.8cm]{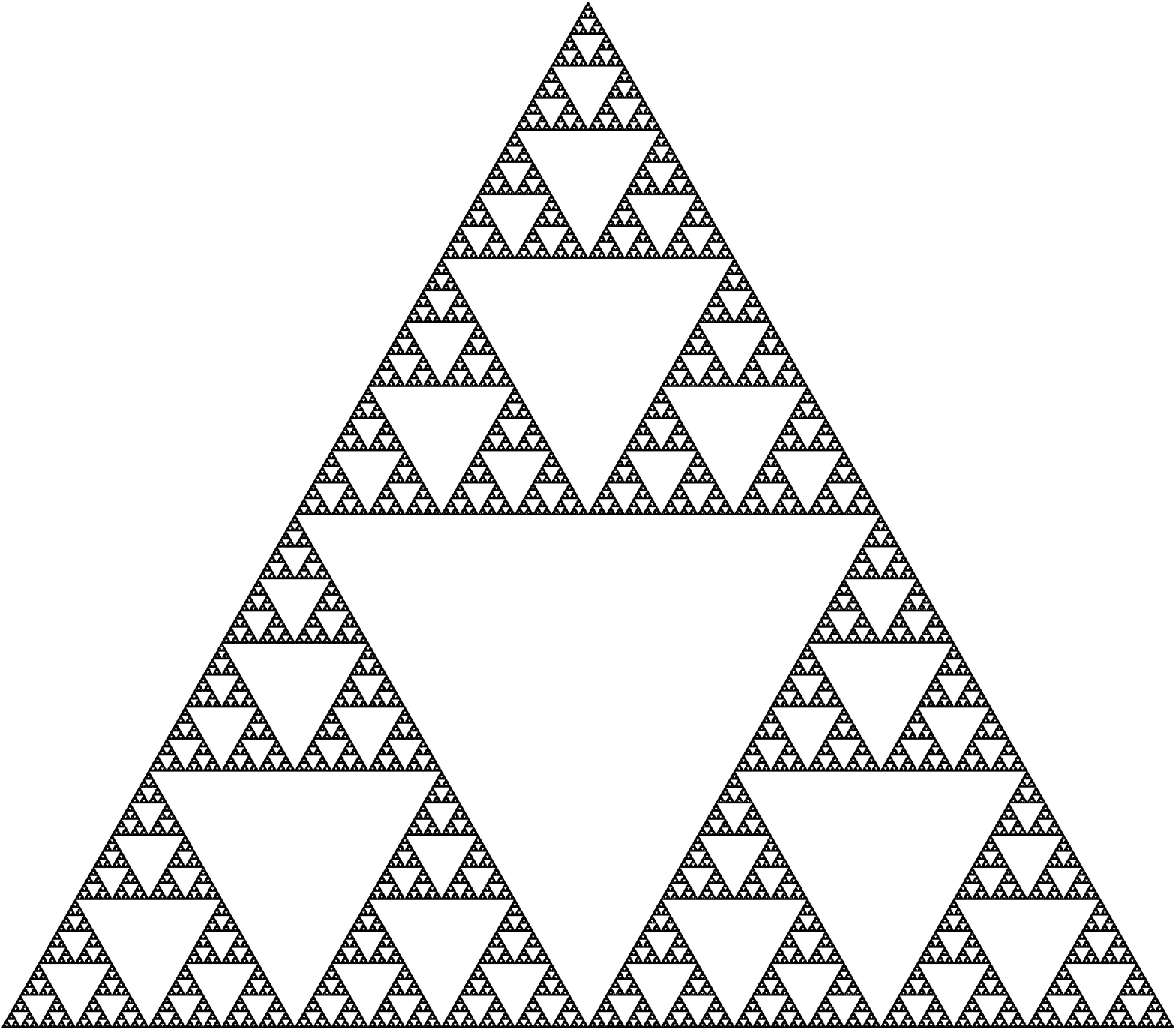}}}
    \qquad
    \subfloat[A self-affine gasket $X_2$ which is not self-similar.]{{\includegraphics[width=5.8cm]{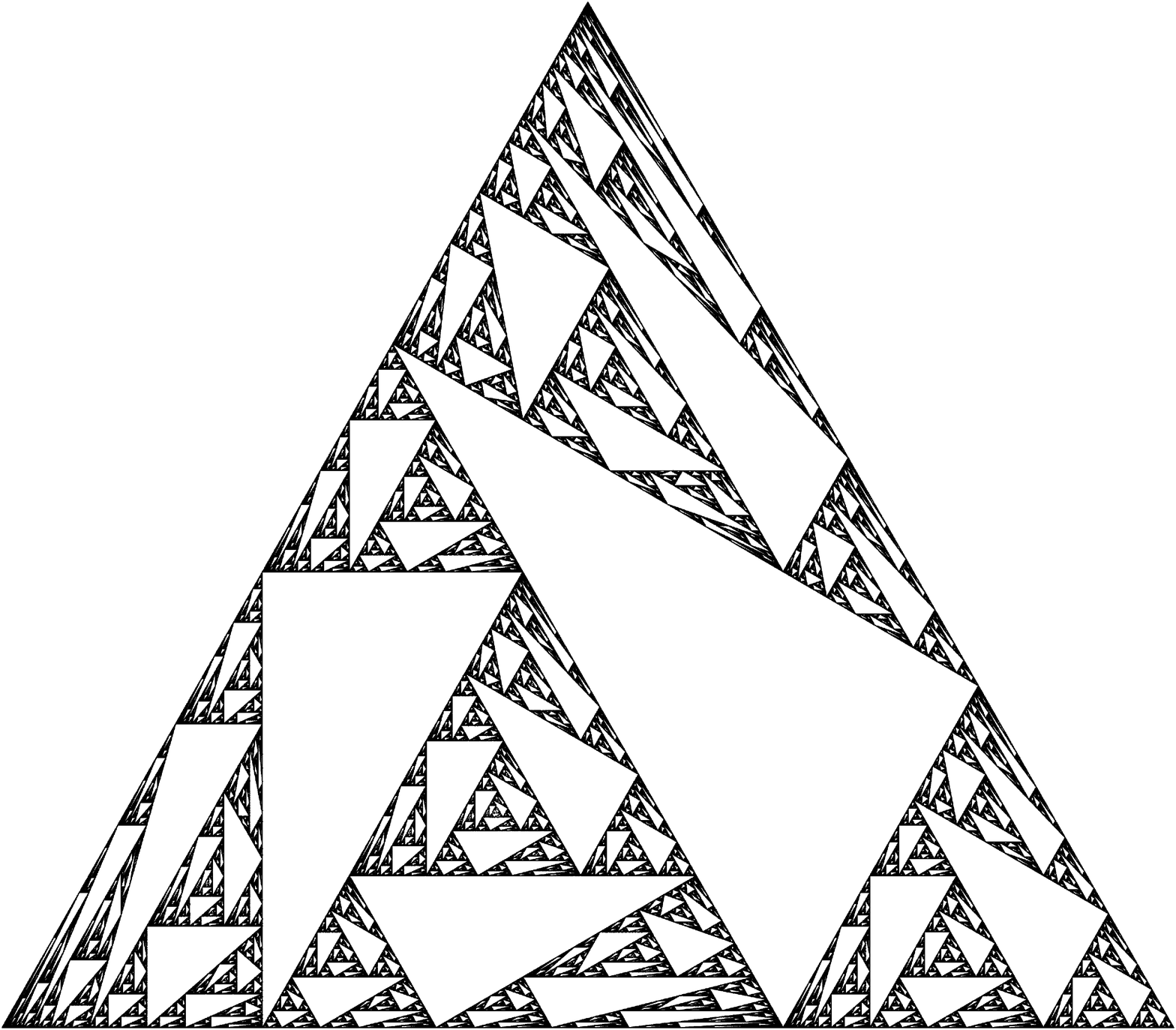}}}
    \caption{A measure of maximal Hausdorff dimension can be constructed on the classical Sierpi\'nski gasket $X_1$ by the simple expedient of giving measure $\frac{1}{3}$ to each of the three copies of $X_1$ with diameter half that of the original, measure $\frac{1}{9}$ to each of the nine sub-copies with diameter $\frac{1}{4}$ that of the original, and so forth. A self-affine gasket which is not self-similar will typically be much less homogenous, and sub-images of the same recursive depth may have very different shapes and sizes. The optimal allocation of measure to the different parts of $X_2$ is believed to be that given by the solution of the variational problem \eqref{eq:var}.}
    \label{fi:onlyfigure}%
\end{figure}

In recent years a number of sufficient conditions have been given for the Hausdorff dimension of a self-affine set to equal the affinity dimension of its defining iterated function system, and the investigation of the equilibrium states of $\varphi^s$ has typically played a critical r\^ole in these works \cite{Ba15,BaKa15,BaRa17,FaKe16,HuLa95,MoSh16}. The equilibrium states of $\varphi^s$ currently being poorly understood in general, a common feature of research which applies these measures to obtain dimension results has been the imposition of conditions on $A_1,\ldots,A_N$ in order to guarantee that the equilibrium states are unique and admit explicit estimates or descriptions \cite{BaKa15,BaKaKo17,BaRa17,FaKe16,HuLa95}. Our objective in this article is to obtain the first complete description and characterisation of the equilibrium states of $\varphi^s$ for invertible affinities in arbitrary dimensions without any assumptions on the matrices $A_i$. We in particular prove:
\begin{theorem}\label{th:intro2}
Let $A_1,\ldots,A_N \in GL_d(\mathbb{R})$ and let $s>0$. Then the number of ergodic equilibrium states of $\varphi^s$ is exactly one if $s \geq d$, is bounded by ${d \choose s}$ if $s$ is an integer and is bounded by ${d \choose \lfloor s\rfloor}{d\choose \lceil s \rceil}$ otherwise. In all cases all equilibrium states are fully supported. \end{theorem}
A detailed description of the structure of the ergodic equilibrium states is complicated to express and is deferred until the following section, but we remark that each ergodic equilibrium state satisfies a so-called \emph{Gibbs inequality} which uniquely characterises it in the space of all $\sigma$-invariant Borel probability measures on $\Sigma_N$. Theorem~\ref{th:intro2} resolves a question of A. K\"aenm\"aki, who asked in \cite{Ka04} whether the number of ergodic equilibrium states is always finite. We note that by standard ergodic decomposition arguments, the set of all equilibrium states associated to fixed $A_1,\ldots,A_N$ and $s$ is precisely the convex hull of the set of ergodic equilibrium states. The case $s \geq d$ can be treated using elementary arguments and the identity $\varphi^s(AB)\equiv \varphi^s(A)\varphi^s(B)$ which holds in this parameter regime: this article therefore focuses on the where case $s \in (0,d)$. 

The equilibrium states of the singular value function $\varphi^s$ in dimension two were fully characterised by D.-J. Feng and A. K\"aenm\"aki \cite{FeKa11} and their ergodic properties investigated thoroughly in \cite{Mo17b,Mo17a}. The case $d=3$ of Theorem~\ref{th:intro2} was proved by A. K\"aenm\"aki and the second named author in \cite{KaMo16}. The case $s\geq d$ is trivial. Examples were constructed in \cite{KaMo16} to show that $(d-\lfloor s\rfloor){d \choose \lfloor s\rfloor}$ distinct ergodic equilibrium states can exist when $s \in (0,d) \setminus \mathbb{Z}$ and that ${d \choose s}$ can exist when $s \in (0,d)\cap \mathbb{Z}$. These lower estimates for the maximum number of equilibrium states were proved in \cite{KaMo16} to be sharp when $d \leq 3$ and conjectured to be sharp in higher dimensions. We therefore do not expect the bound for the number of equilibrium states in Theorem~\ref{th:intro2} to be sharp for non-integer $s \in (0,d)$. An algebraic trick introduced in \cite[\S5]{Mo16} was recently applied in \cite{KaLi17} to bound the number of ergodic equilibrium states when $s \in (0,d) \cap \mathbb{Q}$: if $s-\lfloor s \rfloor = \frac{p}{q}\in\mathbb{Q}$ this gives a bound of ${d \choose \lfloor s\rfloor }^{q-p}{d \choose \lceil s \rceil}^p$, which is clearly weaker than Theorem \ref{th:intro2} for non-integer $s$ and gives no information at all in the case where $s$ is irrational.

As an application Theorem~\ref{th:intro2} permits us to prove the following property of the affinity dimension which was discussed at the start of the introduction:
\begin{theorem}\label{th:aff}
Let $N\geq 2$ and let $T_1,\ldots,T_N \colon \mathbb{R}^d \to \mathbb{R}^d$ be invertible affine contractions. Then
\[\dim_{\mathrm{aff}}(T_1,\ldots,T_{N-1})<\dim_{\mathrm{aff}}(T_1,\ldots,T_{N}).\] 
\end{theorem}
\begin{proof}
Let the contractions $T_i$ be given by $T_ix=A_ix+v_i$ for all $i=1,\ldots,N$ and $x \in \mathbb{R}^d$. The inequality $\dim_{\mathrm{aff}}(T_1,\ldots,T_{N-1})\leq \dim_{\mathrm{aff}}(T_1,\ldots,T_{N})$ follows from the definition and properties of the pressure and affinity dimension so if the conclusion is false then $\dim_{\mathrm{aff}}(T_1,\ldots,T_{N-1})=\dim_{\mathrm{aff}}(T_1,\ldots,T_{N})=s$, say. Trivially $s>0$. In particular there exists a shift-invariant Borel probability measure on $\Sigma_{N-1}$ with Lyapunov dimension equal to $s$, which as noted above is necessarily an equilibrium state of $\varphi^s$ with respect to $(A_1,\ldots,A_{N-1})$. Since $P((A_1,\ldots,A_{N-1}),s)=P((A_1,\ldots,A_N),s)=0$ by the definition of affinity dimension this measure may also be regarded as an equilibrium state of $\varphi^s$ with respect to $(A_1,\ldots,A_N)$ with support $\Sigma_{N-1}\subset \Sigma_N$, but by Theorem~\ref{th:intro2} such a measure must be fully supported on $\Sigma_N$ and this contradiction completes the proof.
\end{proof}
We note that Theorem~\ref{th:aff} is false if the affinities are not assumed to be invertible: for example, if $d=2$ and $\dim_{\mathrm{aff}}(T_1,\ldots,T_N)>1$ but $\mathrm{rank}$ $A_N=1$ then 
it is not difficult to see that $\dim_{\mathrm{aff}}(T_1,\ldots,T_N)=\dim_{\mathrm{aff}}(T_1,\ldots,T_{N-1})$. For further examples see \cite{KaMo16}. Theorem \ref{th:aff} is not difficult to prove for $d=2$ using the results of \cite{FeKa11} and was proved for $d=3$ in \cite{KaMo16}; in the special case where $d$ is arbitrary but $\dim_{\mathrm{aff}}(T_1,\ldots,T_{N})$ is rational, the result was proved in \cite{KaLi17}.

A folklore open problem in the dimension theory of self-affine sets asks under what circumstances the Hausdorff dimension of the attractor of an iterated function system is reduced when one of the transformations $T_i$ is removed. It is clear that this property requires some conditions on the relationship between the different maps $T_i$ in order to avoid trivial counterexamples: for example, if two invertible affine contractions $T_1,T_2$ are given then the two systems $(T_1,T_2)$ and $(T_1,T_2,T_2)$ have unequal affinity dimensions by Theorem~\ref{th:aff} but obviously give rise to the same attractor and no reduction in the Hausdorff dimension can occur when one of the copies of $T_2$ is deleted. To determine complete necessary and sufficient conditions on the maps $T_i$ under which the removal of a transformation reduces the Hausdorff dimension of the attractor thus seems to require a degree of understanding of the relationship between properties of the attractor (such as the Hausdorff dimension) and properties of the iterated function system (such as the affinity dimension). As was remarked earlier in the introduction this relationship is currently far from being understood. We however note the following easy consequence of Theorem~\ref{th:aff} for the Hausdorff dimension of self-affine sets:
\begin{corollary}\label{co:xsapples}
Let $N \geq 2$, let $T_1,\ldots,T_N \colon \mathbb{R}^d\to\mathbb{R}^d$ be invertible affine contractions and let $X$ be the attractor of $(T_1,\ldots,T_N)$. If $\dim_H X$ is equal to $ \dim_{\mathrm{aff}} (T_1,\ldots,T_N)$ then the attractor $X'$ of $(T_1,\ldots,T_{N-1})$ satisfies $\dim_H X'<\dim_H X$.
\end{corollary}
\begin{proof}
We have $\dim_H X' \leq \dim_{\mathrm{aff}} (T_1,\ldots,T_{N-1})$ by a result of Falconer \cite[Proposition 5.1]{Fa88}, $\dim_{\mathrm{aff}}(T_1,\ldots,T_{N-1})<\dim_{\mathrm{aff}}(T_1,\ldots,T_{N})$ by Theorem~\ref{th:aff} and $ \dim_{\mathrm{aff}} (T_1,\ldots,T_N)=\dim_H X$ by hypothesis.
\end{proof}
Affine iterated function systems which meet the hypotheses of Corollary~\ref{co:xsapples} are in a reasonable sense abundant. A theorem of Falconer \cite[Theorem~5.3]{Fa88} asserts that if $A_1,\ldots,A_N$ are linear contractions of $\mathbb{R}^d$ with $\max_i\|A_i\|<\frac{1}{3}$ then for Lebesgue-almost-every $(v_1,\ldots,v_N) \in (\mathbb{R}^d)^N$ the transformations $T_1,\ldots,T_N$ defined by $T_ix\coloneqq A_ix+v_i$ give rise to an attractor with Hausdorff dimension equal to the affinity dimension of $(T_1,\ldots,T_N)$. (The bound $\frac{1}{3}$ was subsequently improved to $\frac{1}{2}$ by B. Solomyak in \cite{So98}.) Related results in which the additive parts of $T_i$ are fixed and the linear parts chosen randomly according to Lebesgue measure were more recently given in \cite{BaKaKo17}. Explicit examples of affine iterated function systems whose attractor has Hausdorff dimension equal to the affinity dimension have been given in \cite{Ba15,BaKa15,FaKe16,HuLa95,MoSh16}. 
Clearly Corollary~\ref{co:xsapples} is also valid when $\dim_H$ is replaced throughout the statement with any other notion of dimension of compact sets which is bounded above by the affinity dimension and is monotone with respect to set inclusion. Whilst the condition $\dim_H X = \dim_{\mathrm{aff}} (T_1,\ldots,T_N)$ is sufficient for the outcome $\dim_H X'<\dim_H X$, it is not necessary: this is discussed further in \S\ref{se:mon} below.

The reader may reasonably ask what difficulties arose in the earlier works \cite{FeKa11,KaMo16} which the present article overcomes. The key difficulty in passing above dimension 2 was essentially as follows. In dimension two the problem reduces to a question concerning the norms $\|A\|^t$ of suitable matrix products $A$ for some fixed parameter $t \in (0,1]$. Since the norm interacts well with the additive structure of $\mathbb{R}^d$ and $M_d(\mathbb{R})$, the obstructions to uniqueness of the equilibrium state arise in terms of the \emph{additive} structures -- that is, linear subspaces of $\mathbb{R}^d$ -- which are preserved by $A_1,\ldots,A_N$. In particular these structures are preserved by the algebra generated by the matrices $A_i$ and lend themselves to the use of linear-algebraic methods. Above two dimensions (or more precisely, when $s \in (1,d-1)\setminus \mathbb{Z}$) one must study directly the quantity $\|A^{\wedge \lfloor s\rfloor}\|^{1+\lfloor s \rfloor-s} \|A^{\wedge\lceil s \rceil}\|^{s-\lfloor s\rfloor}$ which by contrast does not interact in a meaningful way with the additive structure of $\mathbb{R}^d$ or of $M_d(\mathbb{R})$ and necessitates the investigation of structures which are invariant only under the \emph{semigroup} generated by $A_1,\ldots,A_N$. (Here $A^{\wedge k}$ denotes the $k^{\mathrm{th}}$ exterior power of $A$ which will be defined in the following section.) The proper investigation of this quantity therefore requires the use of algebraic geometry in place of linear algebra, and the existence of multiple ergodic equilibrium states is associated with the existence of certain structures which are invariant under the semigroup generated by the matrices $A_i$ but not (in general) under the algebra which they generate. It will be seen that multiplicity of the ergodic equilibrium states of $\varphi^s$ in the parameter range $s \in (1,d-1)\setminus \mathbb{Z}$ is associated with the existence of nontrivial finite invariant subsets of $\mathrm{Gr}(\WEDGE^{\lfloor s \rfloor} \mathbb{R}^d)\times \mathrm{Gr}(\WEDGE^{\lceil s\rceil}\mathbb{R}^d)$, where $\mathrm{Gr}(V)$ denotes the Grassmannian of $V$.

As regards the passage from dimension three to arbitrary dimension, the principal innovation of the preceding article \cite{KaMo16} was a criterion for $A_1,\ldots,A_N$ to have a \emph{unique} equilibrium state with respect to $\varphi^s$. This criterion was combined in \cite{KaMo16} with inherently low-dimensional ``tricks'' which exploited the fact that any nontrivial subspace of $\mathbb{R}^3$ is either one-dimensional or one-codimensional and the fact that $\mathbb{R}^3$ is isomorphic to $\WEDGE^2\mathbb{R}^3$. Between them these methods happened to be sufficient to treat all three-dimensional cases in an ad-hoc manner. Absent from this approach was any method by which to understand those cases in which $A_1,\ldots,A_N$ preserve a subspace or finite union of subspaces with neither dimension nor codimension equal to $1$, a possibility which arises immediately on passage to dimension $4$. This approach also lacked any general mechanism to deal with the case in which neither $A_1^{\wedge \lfloor s \rfloor},\ldots,A_N^{\wedge \lfloor s \rfloor}$ nor $A_1^{\wedge \lceil s \rceil},\ldots,A_N^{\wedge \lceil s \rceil}$ preserve a common invariant subspace but nonetheless more than one ergodic equilibrium state exists, and further did not include a mechanism for handling the case where $A_1^{\wedge \lfloor s \rfloor},\ldots,A_N^{\wedge \lfloor s \rfloor}$ or $A_1^{\wedge \lceil s \rceil},\ldots,A_N^{\wedge \lceil s \rceil}$ preserves a common invariant subspace but $A_1,\ldots,A_N$ does not. In the present article we resolve all three of these issues and in particular elucidate the relationship between the existence of multiple ergodic equilibrium states of $\varphi^s$ and the existence of nontrivial finite invariant subsets of $\mathrm{Gr}(\WEDGE^k \mathbb{R}^d)\times \mathrm{Gr}(\WEDGE^{k+1}\mathbb{R}^d)$.

\section{Statement of technical results}

As in the introduction, for each $N \geq 2$ let $\Sigma_N\coloneqq \{1,\ldots,N\}^{\mathbb{N}}$ which we equip with the infinite product topology, under which $\Sigma_N$ is compact and metrisable. Let $\sigma \colon \Sigma_N \to \Sigma_N$ denote the full shift and $\mathcal{M}_\sigma$ the set of all $\sigma$-invariant Borel probability measures on $\Sigma_N$. In the weak-* topology the set $\mathcal{M}_\sigma$ is compact and metrisable as well as being nonempty. When $N$ is understood we will say that a \emph{word} is a finite sequence of elements of $\{1,\ldots,N\}$. If $\iii=(\iii_k)_{k=1}^n \in \{1,\ldots,N\}^n$ is a word we call $n$ the \emph{length} of the word and write $|\iii|=n$. We let $\Sigma_N^*$ denote the set of all words over the symbols $\{1,\ldots,N\}$. If $\iii$ and $\jjj$ are words we let $\iii\jjj$ denote the word obtained by concatenating $\iii$ and $\jjj$ in the obvious fashion, passing first through the symbols of $\iii$ and then through those of $\jjj$. We note that $\Sigma_N^*$ is a semigroup with respect to concatenation. If $A_1,\ldots,A_N$ are linear transformations of some vector space $V$ then for every $\iii=(\iii_k)_{k=1}^n \in \Sigma_N^*$ we define
\[A_\iii\coloneqq A_{\iii_n}\cdots A_{\iii_2}A_{\iii_1},\]
a notation which will be applied throughout this work. Clearly this defines a semigroup homomorphism from $\Sigma_N^*$ to the space of linear endomorphisms of $V$.

If $x =(x_k)_{k=1}^\infty \in \Sigma_N$ and $n \geq 1$ we let $x|_n$ denote the word $(x_k)_{k=1}^n$. If $\iii \in \Sigma_N^*$ is any word we define
\[[\iii]\coloneqq \left\{(x_k)_{k=1}^\infty \in \Sigma_N \colon x|_{|\iii|}=\iii\right\}\]
which we call a \emph{cylinder set}. Cylinder sets are both closed and open and form a basis for the topology of $\Sigma_N$. In particular any Borel probability measure on $\Sigma_N$ is completely characterised by its values on cylinder sets.

For the purposes of this work we define a \emph{potential} to be any function $\Phi \colon \Sigma_N^* \to [0,+\infty)$. A potential $\Phi$ implicitly defines a sequence of functions $\Phi_n \colon \Sigma_N \to \mathbb{R}$ by $\Phi_n(x)\coloneqq \Phi(x|_n)$. We call a potential \emph{submultiplicative} if $\Phi(\iii\jjj)\leq \Phi(\iii)\Phi(\jjj)$  for every $\iii,\jjj\in\Sigma_N^*$, or equivalently if $\Phi_{n+m}(x) \leq \Phi_n(\sigma^mx)\Phi_m(x)$ for every $n,m\geq 1$ and $x \in \Sigma_N$. If $\Phi_n$ is a submultiplicative potential and $\mu$ an ergodic measure on $\Sigma_N$ then by the subadditive ergodic theorem we have for $\mu$-a.e. $x \in \Sigma_N$
\[\lim_{n \to \infty} \frac{1}{n}\log \Phi_n(x) = \lim_{n \to\infty}\frac{1}{n}\int \log\Phi_n\,d\mu = \lim_{n \to\infty}\frac{1}{n}\sum_{|\iii|=n} \mu([\iii])\log\Phi(\iii).\]
We shall call this limit the \emph{asymptotic average} of $\Phi$ with respect to $\mu$ and denote it by $\Lambda(\Phi,\mu)$. We define the \emph{pressure} of a submultiplicative potential $\Phi$ to be the quantity
\[P(\Phi)\coloneqq \lim_{n \to \infty} \frac{1}{n}\log \sum_{|\iii|=n} \Phi(\iii)\]
and observe that this limit exists by subadditivity. The subadditive variational principle of \cite{CaFeHu08} asserts that if $\Phi$ is any submultiplicative potential then
\begin{equation}\label{eq:svp}P(\Phi)=\sup_{\mu \in \mathcal{M}_\sigma} \left[h(\mu) + \Lambda(\Phi,\mu)\right]\end{equation}
where $h$ denotes Kolmogorov-Sinai entropy. By general considerations involving upper semi-continuity, compactness and convexity this supremum is always attained by at least one ergodic measure, and we call any measure attaining this supremum an \emph{equilibrium state} of $\Phi$. In general multiple ergodic equilibrium states may exist.

We will say that a potential $\Phi$ is \emph{quasimultiplicative} if there exist $\delta>0$ and a finite set $F \subset \Sigma_N^*$ such that
\[\max_{\jjj \in F} \Phi(\iii\jjj\kkk) \geq \delta \Phi(\iii)\Phi(\kkk)\]
for every $\iii,\kkk \in \Sigma_N^*$. If a potential is both submultiplicative and quasimultiplicative then it has a unique equilibrium state which is perforce ergodic. This principle is summarised by the following result of D.-J. Feng (a special case of \cite[Theorem 5.5]{Fe11}) which will be fundamental to our analysis. (Related results may be found in e.g. \cite{Fe04,FeLa02,FeKa11,KaRe14,Ya11}.)
\begin{proposition}[\cite{Fe11}]\label{pr:kare}
Let $N \geq 2$ and let $\Phi \colon \Sigma_N^* \to [0,+\infty)$ be a submultiplicative and quasimultiplicative potential. Then there exists a unique equilibrium state $\mu$ for $\Phi$. Furthermore there exists $C>0$ such that
\begin{equation}\label{eq:gibbs}C^{-1}\Phi(\iii) \leq \frac{\mu([\iii])}{e^{-|\iii|P(\Phi)}}\leq C\Phi(\iii)\end{equation}
for every $\iii \in \Sigma_N^*$ and $\mu$ is the unique $\sigma$-invariant Borel probability measure on $\Sigma_N$ with this property.
\end{proposition}
Henceforth we will say that $\mu$ \emph{satisfies a Gibbs inequality with respect to $\Phi$} if \eqref{eq:gibbs} is satisfied for every $\iii \in \Sigma_N^*$ and for some constant $C>0$ depending only on $\Phi$. If $\mu$ satisfies a Gibbs inequality with respect to $\Phi$, and $\Phi(\iii)>0$ for all $\iii \in \Sigma_N^*$, then clearly $\mu$ is fully supported on $\Sigma_N$. We notice also that in the situation of Proposition \ref{pr:kare} the measure $\mu$ satisfies the approximate submultiplicativity property $\mu([\iii\jjj])\leq C^3\mu([\iii])\mu([\jjj])$ for every $\iii,\jjj \in \Sigma_N^*$ by direct appeal to the Gibbs inequality and the submultiplicativity of $\Phi$.

The principal focus of this article is on potentials of the form $\Phi(\iii)=\varphi^s(A_\iii)$ where $0<s<d$ and $A_1,\ldots,A_N \in \mathbb{R}^d$. Central to our analysis will be a characterisation of the singular value function $\varphi^s$ in terms of exterior algebra. If $1 \leq k \leq d$ we recall that the \emph{$k^{\mathrm{th}}$ exterior power} of $\mathbb{R}^d$ is 
the vector space spanned by formal expressions of the form $u_1\wedge \cdots \wedge u_k$ where $u_1,\ldots,u_k \in \mathbb{R}^d$, subject to the identifications 
\[(\lambda u_1)\wedge u_2 \wedge \cdots \wedge u_k= \lambda(u_1\wedge \cdots \wedge u_k),\]
 \[u_1 \wedge \cdots \wedge u_k = (-1)^{\mathrm{sign}(\varsigma)} u_{\varsigma(1)}\wedge \cdots \wedge u_{\varsigma(k)},\]
 \[(u_1 \wedge \cdots \wedge u_k) + (u_1' \wedge u_2 \wedge \cdots \wedge u_k) = (u_1+u_1')\wedge u_2 \cdots \wedge u_k\]
 where $\lambda \in \mathbb{R}$ and where $\varsigma \colon \{1,\ldots,k\} \to \{1,\ldots,k\}$ is any permutation. If an inner product $\langle \cdot,\cdot\rangle$ on $\mathbb{R}^d$ is understood, then
 \begin{equation}\label{eq:inducednorm}\langle u_1\wedge \cdots \wedge u_k,v_1\wedge \cdots \wedge v_k\rangle \coloneqq  \det [\langle u_i,v_j\rangle]_{i,j=1}^d\end{equation}
 extends by linearity to an inner product on $\WEDGE^k \mathbb{R}^d$. If $u_1,\ldots,u_d$ is a basis for $\mathbb{R}^d$ then the vectors $u_{i_1}\wedge \cdots \wedge u_{i_d}$ such that $1 \leq i_1 < i_2 < \cdots < i_k\leq d$ form a basis for $\WEDGE^k \mathbb{R}^d$, and in particular $\dim \WEDGE^k \mathbb{R}^d={d\choose k}$. If $A \colon \mathbb{R}^d \to \mathbb{R}^d$ is linear then we define $A^{\wedge k}$ to be the unique linear transformation of $\WEDGE^k \mathbb{R}^d$ such that $A^{\wedge k}(u_1\wedge \cdots \wedge u_k)=Au_1\wedge \cdots \wedge Au_k$ for all $u_1,\ldots,u_k \in \mathbb{R}^d$. We have $(A^{\wedge k})^\top=(A^\top)^{\wedge k}$ and $(AB)^{\wedge k}=A^{\wedge k}B^{\wedge k}$ for all linear endomorphisms $A,B$ of $\mathbb{R}^d$. If $A \colon \mathbb{R}^d \to \mathbb{R}^d$ is given and $e_1,\ldots,e_d$ form a basis for $\mathbb{R}^d$ given by (generalised) eigenvectors of $A$ then it is straightforward to check that vectors of the form $e_{i_1}\wedge \cdots \wedge e_{i_k}$ with $1 \leq i_1<\cdots <i_k\leq d$ form a basis for $\WEDGE^k\mathbb{R}^d$ given by (generalised) eigenvectors of $A^{\wedge k}$. It follows from these considerations that $\|A^{\wedge k}\|=\|(A^{\wedge k})^\top A^{\wedge k}\|^{1/2}=\prod_{i=1}^k \alpha_i(A)$ for any linear endomorphism $A$ of $\mathbb{R}^d$ and any $1 \leq k \leq d$, where $\|\cdot\|$ denotes the Euclidean norm implied by the  inner product \eqref{eq:inducednorm}. These considerations yield the characterisation
 \[\varphi^s(A)=\left\|A^{\wedge \lfloor s \rfloor}\right\|^{1+\lfloor s \rfloor -s} \left\|A^{\wedge \lceil s \rceil}\right\|^{s-\lfloor s \rfloor}\]
 for all linear maps $A \colon \mathbb{R}^d \to \mathbb{R}^d$ and $s \in (0,d)$, where we adhere to the conventions $\WEDGE^0\mathbb{R}^d=\mathbb{R}$, $A^{\wedge 0}\equiv\mathrm{id}_{\mathbb{R}}$. This formulation makes the submultiplicativity property $\varphi^s(AB)\leq\varphi^s(A)\varphi^s(B)$ plain. Theorem~\ref{th:intro2} will therefore follow from an investigation of potentials of the form
 \[\Phi(\iii)\coloneqq \left\|A^{\wedge \lfloor s \rfloor}_\iii\right\|^{1+\lfloor s \rfloor -s} \left\|A^{\wedge \lceil s \rceil}_\iii\right\|^{s-\lfloor s \rfloor}\]
for a given $N$-tuple of invertible matrices $A_1,\ldots,A_N$.

We will obtain Theorem~\ref{th:intro2} as a special case of the following more general statement which is the main result of this article. We recall that $GL(V)$ denotes the group of all invertible linear transformations of a (real) finite-dimensional vector space $V$ and that $GL_d(\mathbb{R})\coloneqq GL(\mathbb{R}^d)$.
\begin{theorem}\label{th:intro1}
Let $N \geq 2$ and for each $i=1,\ldots,k$ let $V_i$ be a real vector space with finite dimension $d_i$, let $(A^{(i)}_1,\ldots,A^{(i)}_N)\in GL(V_i)^N$ and let $\beta_i>0$. Define a potential $\Phi \colon \Sigma_N^* \to (0,+\infty)$ by
\[\Phi(\iii)\coloneqq \prod_{i=1}^k \left\|A^{(i)}_\iii\right\|^{\beta_i}\]
for every $\iii \in \Sigma_N^*$. Then $\Phi$ has no more than $\prod_{i=1}^k d_i$ ergodic equilibrium states, and all of its equilibrium states are fully supported.
\end{theorem}
We will see in the proof of Theorem \ref{th:intro1} that the equilibrium states in Theorem \ref{th:intro1} arise as equilibrium states of certain auxiliary potentials $\Phi^{j}$ which are submultiplicative and quasimultiplicative. We may now easily obtain:

\begin{proof}[Proof of Theorem~\ref{th:intro2} assuming Theorem~\ref{th:intro1} ]
The case $s \geq d$ being trivial we assume $0<s<d$. If $s \notin \mathbb{Z}$ then apply Theorem~\ref{th:intro1} with $k=2$, $V_1\coloneqq \WEDGE^{\lfloor s\rfloor}\mathbb{R}^d$, $V_2\coloneqq \WEDGE^{\lceil s\rceil}\mathbb{R}^d$, $A_j^{(1)}\coloneqq A_j^{\wedge \lfloor s \rfloor}$, $A_j^{(2)}\coloneqq A_j^{\wedge \lceil s \rceil}$, $\beta_1\coloneqq  \lceil s \rceil -s $ and $\beta_2\coloneqq  s-\lfloor s\rfloor$. If $s \in \mathbb{Z}$  take $k=1$, $V=\WEDGE^s\mathbb{R}^d$, $A_j^{(1)}\coloneqq A_j^{\wedge s}$ and $\beta_1\coloneqq 1$.
\end{proof}
We note that in principle the definition of the potential $\Phi$ in Theorem~\ref{th:intro1} is sensitive to the choice of norm on $V_i$, but since two potentials $\Phi$, $\Phi'$ which satisfy $C^{-1}\Phi \leq \Phi' \leq C\Phi$ for some constant $C>0$ must have identical equilibrium states by straightforward consideration of the definitions, this consideration has no bearing on the statement of Theorem~\ref{th:intro1} nor on any succeeding result. We shall therefore ignore the precise choice of norms on individual vector spaces and assume henceforth that inner product norms have been chosen arbitrarily but consistently on all spaces being considered.

The proof of Theorem~\ref{th:intro1} falls naturally into two distinct parts, one part dealing with the situation where for every $i=1,\ldots,k$ the matrices $A_1^{(i)},\ldots,A_N^{(i)}$ do not together preserve a proper nonzero linear subspace of $V_i$, and one part passing from this result to the general case. If $B_1,\ldots,B_N \in GL(V)$ let us say that $(B_1,\ldots,B_N)$ is \emph{irreducible} if no proper nonzero linear subspace of $V$ is invariant under every $B_i$, and otherwise let us say that $(B_1,\ldots,B_N)$ is \emph{reducible}. Let us also say that $(B_1,\ldots,B_N)$ is \emph{strongly irreducible} if no finite union of proper nonzero subspaces of $V$ is invariant under every $B_i$. If we define the \emph{orbit under $(B_1,\ldots,B_N)$} of a subspace $U\subseteq V$ to be the set $\{B_\iii U \colon \iii \in \Sigma_N^*\}$ then $(B_1,\ldots,B_N)$ is strongly irreducible if and only if the only subspaces of $V$ with finite orbit under $(B_1,\ldots,B_N)$ are $\{0\}$ and $V$. If $k \geq 1$ is any integer we let $\mathrm{Gr}_k(V)$ denote the set of all $k$-dimensional linear subspaces of $V$.

The following theorem treats the irreducible case of Theorem~\ref{th:intro1}:
\begin{theorem}\label{th:irr}
Let $N \geq 2$ and $k \geq 1$. For each $i=1,\ldots,k$ let $V_i$ be a real vector space with finite dimension $d_i$, let $(A^{(i)}_1,\ldots,A^{(i)}_N)\in GL(V_i)^N$ be irreducible and let $\beta_i>0$. Define a submultiplicative potential $\Phi\colon \Sigma_N^* \to (0,+\infty)$ by
\[\Phi(\iii)\coloneqq \prod_{i=1}^k \left\|A^{(i)}_\iii\right\|^{\beta_i}\]
for every $\iii \in \Sigma_N^*$. For each $i=1,\ldots,k$ let $\ell_i$ be the smallest nonzero integer such that there exists an $\ell_i$-dimensional subspace of $V_i$ which has finite orbit under $(A^{(i)}_1,\ldots,A^{(i)}_N)$. Then $t_i\coloneqq d_i/\ell_i $ is an integer for each $i=1,\ldots,k$. There exist an integer $p \coloneqq  (\max_i t_i)^{-1}\prod_{i=1}^k t_i$ and finite sets $\mathcal{W}_1,\ldots,\mathcal{W}_p \subset \prod_{i=1}^k \mathrm{Gr}_{\ell_i}(V_i)$ with the following properties. For each $j=1,\ldots,p$ the set $\mathcal{W}_j$ is invariant in the sense that $(A_\iii^{(i)}W_i)_{i=1}^k \in \mathcal{W}_j$ for every $(W_i)_{i=1}^k \in \mathcal{W}_j$ and $\iii \in\Sigma_N^*$. 
For each $j=1,\ldots,p$ the potential $\Phi^{j} \colon \Sigma_N^* \to (0,+\infty)$ defined by 
\[\Phi^{j}(\iii)\coloneqq \max_{(W_i)_{i=1}^k \in \mathcal{W}_j} \prod_{i=1}^k \left\|A^{(i)}_\iii|_{W_i}\right\|^{\beta_i}\]
tis submultiplicative and quasimultiplicative and has a unique equilibrium state which satisfies a Gibbs inequality with respect to $\Phi^{j}$. There exists a constant $\tau>0$ such that
\begin{equation}\label{eq:tau}\tau \Phi(\iii) \leq \max_{1 \leq j \leq p} \Phi^{j}(\iii) \leq \Phi(\iii) \end{equation}
for every $\iii \in \Sigma_N^*$.
If $\mu\in \mathcal{M}_\sigma$ is an ergodic equilibrium state of $\Phi$ then it is necessarily an ergodic equilibrium state of at least one of the potentials $\Phi^{j}$. In particular there are not more than $(\max_i t_i)^{-1}\prod_{i=1}^k t_i$ ergodic equilibrium states for $\Phi$ and every equilibrium state of $\Phi$ is fully supported. 
\end{theorem}

We emphasise that Theorem~\ref{th:irr} does not assert that only $p$ finite sets which are invariant under the action of the matrices $A^{(i)}_j$ in the manner described can exist. Rather, it asserts only that there exist $p$ such sets which between them suffice to exhaust the supply of ergodic equilibrium states. For example, if $k=1$, $d_1=2$, $\beta_1=1$ and every $A^{(1)}_j$ is a matrix of rotation through $2\pi/q$ for some odd integer $q >1$ then the hypotheses of Theorem~\ref{th:irr} are satisfied, $p$ is equal to $1$, and uncountably many choices of the finite set $\mathcal{W}_1$ exist, but the measure of maximal entropy is the unique equilibrium state of $\Phi$. Since each candidate for the invariant set $\mathcal{W}_1$ has cardinality exactly $q$ this example also illustrates that the cardinality of each individual set $\mathcal{W}_j$ admits no a priori upper bound. We note that while every ergodic equilibrium state of $\Phi$ is an equilibrium state of one of the potentials $\Phi^{j}$ the converse should not in general be presumed to hold since in certain cases we may have $P(\Phi^{j})<P(\Phi)$ for particular values of $j$.

Let us review some special cases of Theorem~\ref{th:irr}. We note that if enough of the tuples $(A_1^{(i)},\ldots,A_N^{(i)})$ are strongly irreducible then there is a unique equilibrium state of $\Phi$:
\begin{corollary}\label{co:balt}
Let $N$, $k$, $V_i$, $\beta_i$,  $(A_1^{(i)},\ldots,A_N^{(i)})$ and $\Phi$ be as in Theorem~\ref{th:irr} and suppose that $(A_1^{(i)},\ldots,A_N^{(i)})$ is strongly irreducible for at least $k-1$ values of $i$. Then $\Phi$ is quasimultiplicative and has a unique equilibrium state. This equilibrium state satisfies a Gibbs inequality with respect to $\Phi$.
\end{corollary}
\begin{proof}
We have $\ell_i=d_i$ for every $i$ such that $(A^{(i)}_1,\ldots,A^{(i)}_N)$ is strongly irreducible, so $t_i=1$ for at least $k-1$ values of $i$ and it follows immediately that $p=1$. We deduce from \eqref{eq:tau} that $\Phi$ is quasimultiplicative and has a unique equilibrium state, and that this equilibrium state satisfies a Gibbs inequality with respect to $\Phi$.\end{proof}
Taking $k=1$ and $\beta_1\coloneqq s$ in the above corollary yields the invertible case of a result of Feng and K\"aenm\"aki \cite[Proposition 1.2]{FeKa11}:
\begin{corollary}\label{co:ral}
Let $A_1,\ldots,A_N \in GL_d(\mathbb{R})$ and let $s>0$, and suppose that $(A_1,\ldots,A_N)$ is irreducible. Then there exists a unique equilibrium state for $(A_1,\ldots,A_N)$ with respect to the potential $\Phi(\iii)\coloneqq \|A_\iii\|^s$, and this equilibrium state satisfies a Gibbs inequality with respect to $\Phi$.
\end{corollary}
One may also easily obtain a theorem of K\"aenm\"aki and the second named author \cite[Theorem~C]{KaMo16}:
\begin{corollary}
Let $A_1,\ldots,A_N \in GL_d(\mathbb{R})$ and let $0< \ell < s<\ell+1<d$ where $\ell \in \mathbb{Z}$. Suppose that both $(A_1^{\wedge \ell},\ldots,A_N^{\wedge \ell})$ and $(A_1^{\wedge (\ell+1)},\ldots,A_N^{\wedge (\ell+1)})$ are irreducible and that one of them is strongly irreducible. Then there exists a unique equilibrium state for $(A_1,\ldots,A_N)$ with respect to the potential $\Phi(\iii)\coloneqq \varphi^s(A_\iii)$, and this equilibrium state satisfies a Gibbs inequality with respect to $\Phi$.
\end{corollary}
\begin{proof}
Apply Corollary~\ref{co:balt} with $k:=2$, $V_1\coloneqq \WEDGE^\ell\mathbb{R}^d$, $V_2\coloneqq \WEDGE^{\ell+1}\mathbb{R}^d$, $A^{(1)}_j\coloneqq A_j^{\wedge \ell}$ and $A^{(2)}_j\coloneqq A_j^{\wedge(\ell+1)}$ for each $j=1,\ldots,N$, $\beta_1\coloneqq \ell+1-s$ and $\beta_2\coloneqq s-\ell$.
\end{proof}

Let us now state our second main result, which allows Theorem~\ref{th:irr} to be extended beyond the case where each $(A_1^{(i)},\ldots,A^{(i)}_N)$ is irreducible.
Clearly the following result implies Theorem~\ref{th:intro1}.
\begin{theorem}\label{th:force-diagonal}
Let $N \geq 2$ and $k \geq 1$. For each $i=1,\ldots,k$ let $V_i$ be a real vector space with finite dimension $d_i$, let $(A^{(i)}_1,\ldots,A^{(i)}_N)\in GL(V_i)^N$ and let $\beta_i>0$. Define a potential $\Phi \colon \Sigma_N^* \to (0,+\infty)$ by
\[\Phi(\iii)\coloneqq \prod_{i=1}^k \left\|A^{(i)}_\iii\right\|^{\beta_i}\]
for every $\iii \in \Sigma_N^*$.  Then for each $i=1,\ldots,k$ there exist a unique integer $n_i$ and a basis for $V_i$ in which the matrix representation
\begin{equation}\label{eq:mat}A^{(i)}_j =
    \begin{pmatrix}
      A_j^{(i,1)} & * & \cdots & *\\
      0           & A_j^{(i,2)} & \cdots & * \\
      \vdots      & \vdots      & \ddots & \vdots \\
      0           & 0           & \cdots & A_j^{(i,n_i)}
    \end{pmatrix}
\end{equation}
holds for every $j=1,\ldots,N$, where for each fixed $i$ and $r$ the $N$-tuple $(A_1^{(i,r)},\ldots,A_N^{(i,r)})$ consists of square matrices of the same dimension and is irreducible.  If $\mu$ is an ergodic equilibrium state of $\Phi$ then for each $i=1,\ldots,k$ there exists an integer $r_i \in \{1,\ldots,n_i\}$ such that $\mu$ is an equilibrium state of the potential $\Phi^{(r_1,\ldots,r_k)}\colon \Sigma_N^* \to (0,+\infty)$ defined by
\[\Phi^{(r_1,\ldots,r_k)}(\iii)\coloneqq \prod_{i=1}^k \left\|A_\iii^{(i,r_i)}\right\|^{\beta_i}.\]
Every equilibrium state of $\Phi$ is fully supported and the number of ergodic equilibrium states of $\Phi$ is not greater than $(\min_{1 \leq i \leq k} n_i/d_i)\prod_{i=1}^k d_i$. In particular the number of ergodic equilibrium states is not greater than $\prod_{i=1}^k d_i$, and if $\Phi$ has exactly $\prod_{i=1}^k d_i$ ergodic equilibrium states then for every $i=1,\ldots,k$ there exists a basis for $V_i$ in which the matrices $A_1^{(i)},\ldots,A_N^{(i)}$ are all upper triangular.
\end{theorem}

The remainder of this article is structured as follows. In the following section we give an overview of the facts from algebraic geometry which will be required for the proof of Theorem~\ref{th:irr}. In  \S\ref{se:four} we prove a largely algebraic result, Theorem~\ref{pr:quam}, which will be needed in order to demonstrate that the potentials $\Phi^{j}$ defined in Theorem~\ref{th:irr} are quasimultiplicative. In \S\ref{se:five} we apply Theorem~\ref{pr:quam} to prove Theorem~\ref{th:intro1}, and the proof of Theorem~\ref{th:force-diagonal} is given in  \S\ref{se:six}. The optimality or otherwise of the bounds given for the number of ergodic equilibrium states in Theorems~\ref{th:intro2}, \ref{th:irr} and \ref{th:force-diagonal} is investigated in \S\ref{se:seven}, and a possible extension of Corollary \ref{co:xsapples} is discussed in \S\ref{se:mon}.


\section{Review of relevant facts from algebraic geometry}

The proof of Theorem~\ref{th:irr} relies substantially on ideas from elementary algebraic geometry as applied to groups of real invertible linear transformations. Since the majority of researchers in fractal geometry are unlikely to be familiar with these ideas, let us briefly outline the ideas to be employed before beginning the proof of Theorem~\ref{th:irr}. For information relating to affine algebraic varieties and the Zariski topology we refer to \cite[\S2.1.1--\S2.1.4]{OnVi90}; for information relating to real algebraic groups our reference is \cite[\S3.1.1]{OnVi90}.

\subsection{Affine algebraic varieties and the Zariski topology}\label{ss:varieties}

If $V_1,V_2$ are finite-dimensional real vector spaces, a function $p \colon V_1 \to V_2$ will be called a \emph{polynomial} if for some (then for any) bases $u_1,\ldots,u_{d_1}$ of $V_1$ and $v_1,\ldots,v_{d_2}$ of $V_2$ we may write $p$ in the form $p(\sum_{i=1}^{d_1} a_iu_i)=\sum_{i=1}^{d_2} q_i(a_1,\ldots,a_{d_1})v_i$ where each function $q_i \colon \mathbb{R}^{d_1}\to \mathbb{R}$ is a polynomial in the usual sense. For the purposes of this article an \emph{affine algebraic variety} will be any subset $Z$ of a finite-dimensional real vector space $V$ which is equal to the common zero locus of a set of polynomials $V\to \mathbb{R}$. Without loss of generality this set of polynomials may be taken to be finite. The \emph{Zariski topology} on a finite-dimensional real vector space $V$ is defined by declaring a set to be closed if and only if it is an affine algebraic variety. If $Z\subseteq V$ is any affine algebraic variety then we define the Zariski topology on $Z$ to be the subspace topology which it inherits from the Zariski topology on $V$. The Zariski topology on an affine algebraic variety $Z$ satisfies the \emph{descending chain condition}: if $(Z_n)_{n=1}^\infty$ is a sequence of Zariski-closed subsets of $Z$ such that $Z_{n+1}\subseteq Z_n$ for every $n \geq 1$, then $(Z_n)$ must be eventually constant. 

If $Z_1\subseteq V_1$ and $Z_2\subseteq V_2$ are affine algebraic varieties then a function $f \colon Z_1 \to Z_2$ is called a \emph{morphism} if there exists a polynomial $p \colon V_1\to V_2$ such that $p(Z_1)\subseteq Z_2$ and $p|_{Z_1}=f$. Every morphism $Z_1 \to Z_2$ is a continuous function with respect to the Zariski topologies on $Z_1$ and $Z_2$. The \emph{product variety} $Z_1 \times Z_2$ is defined by identifying $Z_1 \times Z_2$ with the corresponding subset of $V_1 \oplus V_2 \simeq \mathbb{R}^{d_1+d_2}$ and equipping it with the Zariski topology which it inherits from $V_1\oplus V_2$. We caution the reader that this topology (called the \emph{Zariski product topology}) is distinct from the ordinary product of the Zariski topologies on $Z_1$ and $Z_2$, having more open sets.

A nonempty Zariski-closed set is normally called \emph{irreducible} if it cannot be written as a finite union of proper nonempty Zariski-closed subsets. To emphasise the difference between this notion of irreducibility and our use of the word to refer to sets of linear maps which do not preserve a common subspace, we will say that a Zariski-closed set $Z$ is an \emph{irreducible variety} if it is not equal to a finite union of Zariski-closed nonempty proper subsets of itself. An important fact which will be used in this article is that every nonempty Zariski-open subset of an irreducible variety is Zariski dense: to see this note that if $U \subset Z=\overline{Z}$ is nonempty, open and not dense then $Z=\overline{U} \cup (Z\setminus U)$ expresses $Z$ as a union of two Zariski-closed nonempty proper subsets of $Z$ and therefore $Z$ is not an irreducible variety. (Here closures are of course taken in the Zariski topology.)

 Every Zariski-closed set $Z$ may be written in a unique way as a finite union of irreducible varieties $Z_1,\ldots,Z_k$ each of which is maximal in the sense that it is not properly contained in any irreducible subvariety of $Z$. The sets $Z_i$ are referred to as the \emph{irreducible components} of $Z$. In general the irreducible components of a Zariski-closed set may intersect (for example, if $Z$ is the union of two overlapping circles in $\mathbb{R}^2$). It is not difficult to check that if $f \colon Z \to Z$ is a homeomorphism in the Zariski topology then $f(Z_i)$ is also an irreducible component of $Z$ for every irreducible component $Z_i$.

\subsection{Real algebraic groups}\label{ss:alg-groups}

A \emph{real algebraic group} is a group $G$ endowed with the structure of a real affine algebraic variety such that the map $g \mapsto g^{-1}$ defines a morphism $G \to G$ and the map $(g_1,g_2)\mapsto g_1g_2$ defines a morphism $G \times G \to G$. If $G$ is a real algebraic group and $g_0 \in G$ is fixed then the maps $g \mapsto g_0g$ and $g \mapsto gg_0$ define Zariski homeomorphisms of $G$, a fact which will be used repeatedly in the remainder of this article.

The real algebraic groups considered in this article will all arise as Zariski-closed subgroups of $GL(V)$ where $V$ is some finite-dimensional real vector space. Importantly $GL(V)$ itself has the structure of a real algebraic group, which does not arise directly from its definition but via the following contrivance. Let us identify $V$ with $\mathbb{R}^d$. We identify $GL_d(\mathbb{R})$ with the set
\[\left\{\begin{pmatrix}A&0\\ 0&x\end{pmatrix} \in M_{d+1}(\mathbb{R}) \colon A \in M_d(\mathbb{R}), x \in \mathbb{R}\text{ and }x\cdot(\det A)=1\right\}.\]
This defines an affine subvariety of $M_{d+1}(\mathbb{R})$ which is a real algebraic group with respect to the standard operations of matrix multiplication and inversion, and there is a group isomorphism from $GL_d(\mathbb{R})$ to the above group given by $A \mapsto A \oplus (\det A)^{-1}$. Via this identification a function $p \colon GL_{d_1}(\mathbb{R}) \to \mathbb{R}^{d_2}$ is thus considered to be a polynomial if and only if each entry of the vector $p(A)$ is a polynomial function of the entries of the matrix $A$ and of the variable $x=1/\det A$. Thus $GL_d(\mathbb{R})$ equipped with this structure of polynomial functions meets the definition of a real algebraic group. We note that the polynomial structure on $GL(V)$ is independent of the basis used in identifying $V$ with $\mathbb{R}^d$.  
It is not difficult to see that any Zariski-closed subgroup of $GL(V)$ is also a real algebraic group. For concreteness the reader may find it helpful to know that every group of this type is a Lie group, although this fact will not be used.

An important principle in this work is that the Zariski closure of any \emph{subsemigroup} of $GL_d(\mathbb{R})$ is a real algebraic group. To see this one first shows that the Zariski closure of a subsemigroup of $GL_d(\mathbb{R})$ is also a semigroup. The descending chain condition now implies that a Zariski-closed subsemigroup $H$ of $GL_d(\mathbb{R})$ is a group by the following argument: if $g \in H$ then the sequence of sets $g^{-n}H$ forms a descending chain of closed sets which eventually terminates, so $g^{-n}H=g^{-n-1}H$ for some $n$. Hence $H=g^{-1}H$ and therefore $\mathrm{id} =g^{-1}g \in g^{-1}H=H$. Hence $g^{-1}=g^{-1}\cdot\mathrm{id} \in g^{-1}H=H$. It follows that $H$ is a group as claimed. 

We recall that a homomorphism $\rho$ from an abstract group $G$ to $GL(V)$ is called an \emph{irreducible representation} if there does not exist a proper nonzero linear subspace of $V$ which is preserved by every element of $\rho(G)$. A homomorphism from a real algebraic group $G \leq GL(V)$ to $GL(W)$ is called \emph{regular} if it is a morphism of affine algebraic varieties in addition to being a homomorphism.

\subsection{Components of real algebraic groups}\label{ss:comp}

One may show that if $G$ is a real algebraic group then exactly one of its irreducible components contains the identity, and this component is denoted $G^\circ$. For every $g \in G$ the left and right cosets $gG^\circ$ and $G^\circ g$ are also irreducible components of $G$, a fact which follows from the fact that left and right multiplication by $g$ induce homeomorphisms of $G$ in the Zariski topology. One may show that $G^\circ$ is a subgroup of $G$. Since for each $g \in G$ the set $gG^\circ g^{-1}$ is an irreducible component of $G$ which contains the identity, it equals $G^\circ$, and therefore $G^\circ $ is a normal subgroup of $G$. Since $G$ has finitely many irreducible components the quotient group $G/G^\circ$ is finite.

If $G_1,G_2$ were distinct irreducible components of $G$ which shared a common element $g \in G_1 \cap G_2$ then $g^{-1}G_1$ and $g^{-1}G_2$ would be distinct irreducible components which contain the identity, contradicting the uniqueness of $G^\circ$. It follows that the distinct irreducible components of $G$ are pairwise disjoint, and since they are finite in number each irreducible component of $G$ must be Zariski open as well as Zariski closed.

\section{A  quasimultiplicativity result}\label{se:four}

Before commencing the proof of Theorem~\ref{th:irr} we prove the following result which will be used to establish the submultiplicativity and quasimultiplicativity of the potentials $\Phi^{j}$.
\begin{theorem}\label{pr:quam}
Let $V$ be a finite-dimensional real vector space, $G\leq GL(V)$ a real algebraic group and $H\subset G$ a subsemigroup which is Zariski-dense in $G$. Let $k \geq 1$. For each $i=1,\ldots,k$ let $V_i$ be a finite-dimensional real vector space, $\rho_i \colon G \to GL(V_i)$ a regular irreducible representation, and $\beta_i>0$ a real number. For each $i$ let $U_i\subseteq V_i$ be a nonzero vector space which is preserved by $\rho_i(G^\circ)$ and has the smallest dimension of any such subspace. Define $\ell_i\coloneqq \dim U_i$ for each $i=1,\ldots,k$. Define
\[\mathcal{W}\coloneqq \left\{(\rho_i(g)U_i)_{i=1}^k \colon g \in G\right\}\subset \prod_{i=1}^k \mathrm{Gr}_{\ell_i}(V_i).\]
Then $\mathcal{W}$ is finite. Define also
\[\phi(g)\coloneqq \max_{(W_i)_{i=1}^k \in \mathcal{W}}  \prod_{i=1}^k \|\rho_i(g)|_{W_i}\|^{\beta_i}\]
for every $g \in G$. Then $\phi(g_1g_2)\leq \phi(g_1)\phi(g_2)$ for every $g_1,g_2 \in G$. Furthermore there exist $\delta>0$ and a finite set $H'\subset H$ such that for every $g_1,g_2 \in G$ we may find $h \in H'$ such that $\phi(g_1hg_2)\geq \delta \phi(g_1)\phi(g_2)$.
\end{theorem}
\begin{proof}
We will show in a moment that $\mathcal{W}$ is finite. This property being assumed, let us first show that $\phi(g_1g_2)\leq \phi(g_1)\phi(g_2)$ for all $g_1,g_2 \in G$. Given $g_1$ and $g_2$, choose $(W_i)_{i=1}^k=(\rho_i(g_0)U_i)_{i=1}^k \in \mathcal{W}$ such that
\[\phi(g_1g_2)= \prod_{i=1}^k \left\|\rho_i(g_1g_2)|_{W_i}\right\|^{\beta_i},\]
then we have
\begin{align*} \phi(g_1g_2)&= \prod_{i=1}^k \|\rho_i(g_1g_2)|_{W_i}\|^{\beta_i}\\
& \leq  \left(\prod_{i=1}^k \left\|\rho_i(g_1)|_{\rho_i(g_2)W_i}\right\|^{\beta_i}\right)\left( \prod_{i=1}^k \left\|\rho_i(g_2)|_{W_i}\right\|^{\beta_i}\right)\leq \phi(g_1)\phi(g_2)\end{align*}
as required since $(\rho_i(g_2)W_i)_{i=1}^k =(\rho_i(g_2g_0)U_i)_{i=1}^k\in \mathcal{W}$ by definition. 

The remainder of the proof proceeds through a series of lemmas:
\begin{lemma}\label{le:loccon}
The function $g \mapsto (\rho_i(g)U_i)_{i=1}^k$ is constant on each irreducible component of $G$.
\end{lemma}
\begin{proof}
Let $G_j$ be an irreducible component of $G$ and suppose that $g_1,g_2 \in G_j$. The set $g_2^{-1}G_j$ is an irreducible component of $G$ (since $g \mapsto g_2^{-1}g$ is a Zariski homeomorphism of $G$) and contains the identity since $g_2 \in G_j$, so $g_2^{-1}G_j=G^\circ$ and hence in particular $g_2^{-1}g_1 \in G^\circ$. It follows from the definition of $(U_i)_{i=1}^k$ that $(\rho_i(g_2^{-1}g_1)U_i)_{i=1}^k = (U_i)_{i=1}^k$ and therefore $(\rho_i(g_2)U_i)_{i=1}^k = (\rho_i(g_1)U_i)_{i=1}^k$ as required.
\end{proof}
Let us now show that any element of $\mathcal{W}$ may be mapped onto any other by the action of some element of $H$: 
\begin{lemma}\label{le:switch}
The set $\mathcal{W}$ is finite, and there exists a finite set $H_0 \subseteq H$ such that given any $(W_i)_{i=1}^k,(W_i')_{i=1}^k \in \mathcal{W}$ we may find $h \in H_0$ such that $(W_i')_{i=1}^k =(\rho_i(h)W_i)_{i=1}^k$. 
\end{lemma}
\begin{proof}
Let $G_1,\ldots,G_\ell$ denote the irreducible components of $G$ and recall from \S\ref{ss:comp} that each is Zariski open as well as Zariski closed. By the previous lemma it follows that $\mathcal{W}$ has at most $\ell$ elements and in particular is finite. To prove the remainder of the lemma it suffices to show that for every $(W_i)_{i=1}^k,(W_i')_{i=1}^k \in \mathcal{W}$ the set
\[\mathcal{U} \coloneqq \left\{g \in G \colon (W_i')_{i=1}^k =(\rho_i(g)W_i)_{i=1}^k \right\}\]
contains an element of $H$. Since $H$ is Zariski-dense in $G$ it is sufficient to show that $\mathcal{U}$ contains a nonempty Zariski-open set. By the definition of $\mathcal{W}$ it follows that we may choose $g_2 \in G$ such that $(\rho_i(g_2)U_i)_{i=1}^k=(W_i)_{i=1}^k$. Now note that
\[\mathcal{V}\coloneqq \left\{g \in G \colon (W_i')_{i=1}^k =(\rho_i(g)U_i)_{i=1}^k \right\}\]
is nonempty by definition of $\mathcal{W}$, satisfies $\mathcal{V} g_2^{-1} \subseteq \mathcal{U}$  and by the previous lemma contains at least one irreducible component $G_j$ of $G$. We have $G_j g_2^{-1} \subseteq \mathcal{V} g_2^{-1} \subseteq \mathcal{U}$ and therefore $\mathcal{U}$ contains a nonempty Zariski-open set. In particular $\mathcal{U}$ contains an element of $H$ as required. 
\end{proof}

\begin{lemma} 
Let $(W_i)_{i=1}^k, (W_i')_{i=1}^k \in \mathcal{W}$ and for each $i=1,\ldots,k$ let $C_i \colon U_i \to W_i'$ and $D_i\colon W_i \to U_i$ be nonzero linear maps. Then there exists $h \in H$ such that $\rho_i(h)U_i=U_i$ and $C_i(\rho_i(h)|_{U_i})D_i \neq 0$ for every $i=1,\ldots,k$. \end{lemma}
\begin{proof}
We claim that for each $i=1,\ldots,k$ the set
\[\mathcal{U}_i\coloneqq \left\{g \in G^\circ \colon C_i(\rho_i(g)|_{U_i})D_i \neq 0 \right\}\]
is nonempty. If this is not the case for some $i$ then the vector space
\[\hat{U}_i\coloneqq \mathrm{span}\left(\bigcup_{g \in G^\circ} (\rho_i(g)|_{U_i})D_iW_i\right)\]
is a linear subspace of the kernel of $C_i$ and hence is a proper subspace of $U_i$ since $C_i$ is not the zero map. On the other hand it is also not the zero space since it contains $D_iW_i$ and $D_i$ is not the zero map. Lastly it is clear that $\rho_i(g)\hat{U}_i=\hat{U}_i$ for every $g \in G^\circ$. It follows that $\hat{U}_i$ is a nonzero proper vector subspace of $U_i$ which is invariant under $\rho_i(G^\circ)$. This contradicts the definition of $U_i$, and it follows that $\mathcal{U}_i$ is nonempty as claimed. Clearly each $\mathcal{U}_i$ is Zariski open, and $G^\circ$ is an irreducible variety. We recall from \S\ref{ss:varieties} that every nonempty Zariski-open subset of an irreducible variety is dense, so $\bigcap_{i=1}^k \mathcal{U}_i \neq \emptyset$ and since $H$ is Zariski dense in $G$ there exists $h \in H \cap \bigcap_{i=1}^k \mathcal{U}_i$ which proves the lemma.
\end{proof}
\begin{lemma}\label{le:interp}
There exist a finite set $H_1 \subset H$ and a real number $\delta_0>0$ with the following property: given any $(W_i)_{i=1}^k, (W_i')_{i=1}^k\in \mathcal{W}$,  if $g_1,g_2 \in G$ satisfy 
$\rho(g_1)U_i=W_i'$ and $\rho(g_2)W_i=U_i$ for all $i=1,\ldots,k$, then there exists $h \in H_1$ such that $\rho_i(h)U_i=U_i$ for all $i=1,\ldots,k$ and
\[\prod_{i=1}^k \left\|\rho_i(g_1hg_2)|_{W_i}\right\|^{\beta_i} \geq \delta_0\left(\prod_{i=1}^k \left\|\rho_i(g_1)|_{U_i}\right\|^{\beta_i}\right)\left(\prod_{i=1}^k \left\|\rho_i(g_2)|_{W_i}\right\|^{\beta_i}\right).\]
\end{lemma}
\begin{proof}
Since $\mathcal{W}$ is finite it is clearly sufficient to establish the existence of $H_1$ and $\delta_0$ for fixed $(W_i)_{i=1}^k, (W_i')_{i=1}^k\in \mathcal{W}$ and we therefore fix these objects throughout the proof. We claim that the following stronger property is satisfied: there exist a finite set $H_1 \subset H$ and a real number $\delta_0>0$ such that if for every $i=1,\ldots,k$ we are given arbitrary linear maps $C_i \colon U_i \to W_i'$ and $D_i \colon W_i \to U_i$, then there exists $h \in H_1$ such that $\rho_i(h)U_i=U_i$ for all $i=1,\ldots,k$ and
\[\prod_{i=1}^k \left\|C_i(\rho_i(h)|_{U_i})D_i\right\|^{\beta_i} \geq \delta_0\left(\prod_{i=1}^k \|C_i\|^{\beta_i}\right)\left(\prod_{i=1}^k \|D_i\|^{\beta_i}\right).\]
Applying this claim with $C_i=\rho_i(g_1)|_{U_i}$ and $D_i=\rho_i(g_2)|_{W_i}$ will then suffice to prove the lemma.

Let us prove this claim. By homogeneity we may assume $\|C_i\|=\|D_i\|=1$ for every $i=1,\ldots,k$. By the compactness of the unit spheres of the vector spaces of all linear maps $U_i \to W_i'$ and $W_i \to U_i$ the claim follows if we can establish the following result: given nonzero linear maps $C_i \colon U_i \to W_i'$ and $D_i \colon W_i \to U_i$ there exists $h \in H$ such that $\rho_i(h)U_i=U_i$ for all $i=1,\ldots,k$ and
\[\prod_{i=1}^k \left\|C_i(\rho_i(h)|_{U_i})D_i\right\|^{\beta_i} \neq 0.\]
But this is precisely the previous lemma.
\end{proof}
We may now complete the proof of Theorem~\ref{pr:quam}. Let $\delta_0>0$ and $H_1 \subset H$ be as in the previous lemma. Given $g_1, g_2 \in G$, choose $(W_i)_{i=1}^k, (W_i')_{i=1}^k, (W_i'')_{i=1}^k$ and $(W_i''')_{i=1}^k \in\mathcal{W}$ such that
\[\phi(g_1)=\prod_{i=1}^k \left\|\rho_i(g_1)|_{W_i''}\right\|^{\beta_i}, \qquad \phi(g_2)=\prod_{i=1}^k \left\|\rho_i(g_2)|_{W_i}\right\|^{\beta_i}\]
and $\rho_i(g_1)W_i''=W_i'''$ and $\rho_i(g_2)W_i=W_i'$ for each $i=1,\ldots,k$. Define
\[\kappa\coloneqq \min_{h \in H_0} \prod_{i=1}^k \left\|\rho_i(h)^{-1}\right\|^{-\beta_i}>0\]
where $H_0$ is as defined in Lemma~\ref{le:switch}. Using Lemma~\ref{le:switch} choose $h_1,h_2 \in H_0 $ such that $\rho_i(h_2)W_i'=U_i$ and $\rho_i(h_1)U_i=W_i''$ for all $i=1,\ldots,k$. Then
\[\prod_{i=1}^k \|\rho_i(g_1h_{1})|_{U_i}\|^{\beta_i} \geq \kappa\prod_{i=1}^k \|\rho_i(g_1)|_{W_i''}\|^{\beta_i}=\kappa \phi(g_1), \]
\[\prod_{i=1}^k \|\rho_i(h_{2}g_2)|_{W_i}\|^{\beta_i} \geq \kappa\prod_{i=1}^k \|\rho_i(g_2)|_{W_i}\|^{\beta_i}=\kappa\phi(g_2) \]
and $\rho(g_1h_{1})U_i=W_i'''$ and $\rho(h_{2}g_2)W_i=U_i$ for every $i=1,\ldots,k$. It follows from Lemma~\ref{le:interp} that we may find $h_0 \in H_1$ such that 
\begin{align*}
\phi(g_1h_1h_0h_2g_2)&\geq \prod_{i=1}^k \|\rho_i(g_1h_1h_0h_2g_2)|_{W_i}\|^{\beta_i}\\
&\geq \delta_0\left(\prod_{i=1}^k \|\rho_i(g_1h_{1})|_{U_i}\|^{\beta_i} \right)\left(\prod_{i=1}^k \|\rho_i(h_{2}g_2)|_{W_i}\|^{\beta_i}\right)\\
 &\geq \kappa^2\delta_0\phi(g_1)\phi(g_2),\end{align*}
and we have proved the theorem with $\delta\coloneqq \kappa^2\delta_0$ and $H'\coloneqq H_0H_1H_0$.\end{proof}


\section{The irreducible case: Proof of Theorem~\ref{th:irr}}\label{se:five}

Let $d\coloneqq \sum_{i=1}^k d_i$ and $V\coloneqq \bigoplus_{i=1}^k V_i$. Define $A_j\coloneqq \bigoplus_{i=1}^k A_j^{(i)} \in GL(V)$ for each $j=1,\ldots,N$, let $H\subset  GL(V)$ be the semigroup generated by $A_1,\ldots,A_N$ and let $G$ denote the Zariski closure of $H$ in $GL(V)$, which is obviously contained in the direct product of the $k$ groups $GL(V_i)$. For each $i$ let $\rho_i \colon G \to GL(V_i)$ be given by restriction to $V_i$ in the obvious fashion so that $\rho_i(A_\iii)=A_\iii^{(i)}$ for each $i=1,\ldots,k$ and $\iii \in \Sigma_N^*$. As was remarked in \S\ref{ss:alg-groups}, $G$ is necessarily a real algebraic group. By hypothesis the subsemigroup of $GL(V_i)$ generated by $A^{(i)}_1,\ldots,A^{(i)}_N$ does not preserve a proper nonzero subspace of $V_i$. Since this semigroup is contained in  $\rho_i(G)$ it follows that each $\rho_i$ is an irreducible representation of $G$. Clearly each $\rho_i$ is regular. 

We begin by showing that for each $i=1,\ldots,k$ the smallest possible dimension of a nonzero invariant subspace of the group $\rho_i(G^\circ)$ is precisely the number $\ell_i$ appearing in the statement of Theorem~\ref{th:irr}. The following argument recalls a classical observation of Gol'dsheid and Margulis \cite[Lemma 6.2]{GoMa89}:
\begin{lemma}\label{le:ell-i-minimal}
For each $i \in \{1,\ldots,k\}$ let $U_i\subseteq V_i$ be an $\ell_i$-dimensional subspace which has finite orbit under $A^{(i)}_1,\ldots,A^{(i)}_N$. Then $\rho_i(g)U_i=U_i$ for every $g \in G^\circ$ and no nonzero subspace of $V_i$ with dimension less than $\ell_i$ has this property. Furthermore the set $\{\rho_i(g)U_i\colon g \in G\}$ is finite for each $i$.
\end{lemma}
\begin{proof}
Fix $i$ and $U_i$ throughout the proof and observe that $U_i$ has finite orbit under $\rho_i(H)$. We claim that  $\rho(g)U_i=U_i$ for every $g \in G^\circ$. Fix $i$ and let $U^1_i,\ldots,U^{T_i}_i$ be a complete list of the distinct images of $U_i$ under $\rho(H_i)$. We wish to show that this list also exhausts the possible images of $U_i$ under $\rho_i(G)$. If $\ell_i=d_i$ then this holds trivially, so let us assume $\ell_i<d_i$. Fix an inner product $\langle\cdot,\cdot\rangle$ on $V_i$, let $u_1,\ldots,u_{\ell_i}$ be a basis for $U_i$ and for each $j=1,\ldots,T_i$ let $v_{\ell_i+1}^j,\ldots,v_{d_i}^j$ be a basis for $(U_i^j)^\perp$. For each $j=1,\ldots,T_i$ define
\[\mathcal{V}_j\coloneqq \left\{g \in G \colon \rho_i(g)U_i=U^j_i\right\}.\]
We have
\[\mathcal{V}_j=\left\{g \in G \colon \langle \rho_i(g)u_n,v_m^j\rangle =0\text{ for all }1\leq n \leq \ell_i\text{ and }\ell_i+1\leq m \leq d_i\right\}\]
and this set is Zariski closed since $\rho_i$ is regular. We have $H=\bigcup_{j=1}^{T_i} (H \cap \mathcal{V}_j)$ and since $H$ is Zariski dense in $G$ we obtain $G = \overline{H}= \bigcup_{j=1}^{T_i} \overline{H \cap \mathcal{V}_j}\subseteq \bigcup_{j=1}^{T_i}\mathcal{V}_j\subseteq G$ where closures are taken in the Zariski topology. The resulting equation $G=\bigcup_{j=1}^{T_i} \mathcal{V}_j$ demonstrates that $\{\rho_i(g)U_i\colon g \in G\}$ is equal to the finite set $\{U_i^1,\ldots,U_i^{T_i}\}=\{\rho_i(h)U_i \colon h \in H\}$ as desired. We deduce the equation $G^\circ = \bigcup_{j=1}^{T_i} (G^\circ \cap \mathcal{V}_j)$ which expresses $G^\circ$ as a finite union of Zariski-closed subsets, and since $G^\circ$ is an irreducible variety it follows that $G^\circ= G^\circ \cap \mathcal{V}_j$ for some $j$. For this $j$ we have $\rho_i(g)U_i=U^j_i$ for all $g \in G^\circ$ and since $G^\circ$ contains the identity we conclude that $\rho_i(g)U_i=U_i$ for all $g \in G^\circ$ as required.

It remains to show that there is no nonzero subspace $\hat{U}_i\subseteq V_i$ which is fixed by $\rho_i(G^\circ)$ and has dimension strictly smaller than $\ell_i$. Suppose for a contradiction that such a space exists. If $g_1,g_2 \in G_j$ belong to the same irreducible component of $G$ then by identical reasoning to the proof of Lemma~\ref{le:loccon} we have $g_2^{-1}g_1 \in G^\circ$ and therefore $\rho_i(g_2^{-1}g_1)\hat{U}_i=\hat{U}_i$ so that $\rho_i(g_1)\hat{U}_i=\rho_i(g_2)\hat{U}_i$, and it follows that the map $g \mapsto \rho_i(g)\hat{U}_i$ is constant on each irreducible component of $G$. Since $G$ has finitely many irreducible components it follows that the orbit of $\hat{U}_i$ under $\rho_i(G)$, and hence under $\rho_i(H)$, is finite; but then the orbit of $\hat{U}_i$ under $(A_1^{(i)},\ldots,A_N^{(i)})$ is finite, and this contradicts the minimality of $\ell_i$. \end{proof}

For the remainder of the proof let us fix $U_1,\ldots,U_k$ such that $\rho_i(G^\circ)U_i=U_i$ and $\dim U_i=\ell_i$. We make the following observation:
\begin{lemma}\label{le:fix}
For every $i=1,\ldots,k$ and $g_0 \in G$ the subspace $\rho_i(g_0)U_i$ is fixed by every element of $\rho_i(G^\circ)$. 
\end{lemma}
\begin{proof}
The sets $g_0G^\circ$ and $G^\circ g_0$ are both irreducible components of $G$ which contain $g_0$, therefore they are identical as noted at the end of \S\ref{ss:comp}. If $g \in G^\circ$ then we have $gg_0=g_0g'$ for some $g' \in G^\circ$, so $\rho_i(g)\rho_i(g_0)U_i=\rho_i(g_0)\rho_i(g')U_i=\rho_i(g_0)U_i$ as required.
\end{proof}
We wish next to show that for each $i$ the vector space $V_i$ may be written as a direct sum of images of $U_i$:
\begin{lemma}\label{le:decomp}
For each $i=1,\ldots,k$ there exist $t_i\in\mathbb{N}$ and subspaces $U_i^1,\ldots,U_i^{t_i} \in \{\rho_i(g)U_i\colon g \in G\}$ such that $U_i^1=U_i$ and $V_i=\bigoplus_{j=1}^{t_i}U_i^j$. In particular $t_i\ell_i=d_i$ for every $i=1,\ldots,k$.
\end{lemma}
\begin{proof}
Fix $i$ and let $d_i'\leq d_i$ be the largest integer such that we may form a direct sum $V_i'= U_i^1 \oplus \cdots \oplus U_i^{t_i'}\subseteq V_i$ with dimension $d_i'$ where $U_i^1=U_i$ and $U_i^j \in \{\rho_i(g)U_i\colon g \in G\}$ for each $j$. Note that by the previous lemma any such sum is a $\rho_i(G^\circ)$-invariant subspace of $V_i$. Clearly $d_i' \geq \ell_i>0$ since $U_i$ itself is such a direct sum. Suppose for a contradiction that $d_i' < d_i$, then for every $g \in G$ we have $\rho_i(g)U_i \cap V_i' \neq \{0\}$ since otherwise $d_i'$ would not be maximal. If there exists $g \in G$ such that $0<\dim (\rho_i(g)U_i \cap V_i')<\dim U_i$ then by the previous lemma $\rho_i(g)U_i \cap V_i'$ is a subspace of $V_i$ which is fixed by every element of $\rho_i(G^\circ)$ but has dimension strictly smaller than $\ell_i$, contradicting Lemma~\ref{le:ell-i-minimal}. Otherwise $\dim (\rho_i(g)U_i \cap V_i')=\dim U_i$ for every $g \in G$ which implies that $\rho_i(g)U_i \subseteq V_i'$ for every $g \in G$. It follows then that the vector space $\mathrm{span}\bigcup_{g \in G}\rho_i(g)U_i\subseteq V_i'$ is fixed by $\rho_i(g)$ for every $g \in G$ but has dimension not greater than $d_i'<d_i$, contradicting the irreducibility of the representation $\rho_i$. We conclude that the inequality $d_i'<d_i$ is impossible and therefore the integer $t_i$ and subspaces $U_i^j$ such that $V_i=\bigoplus_{j=1}^{t_i}U_i^j$ must exist. The equation $d_i=t_i\ell_i$ follows trivially.
\end{proof}
For the remainder of the proof we fix for each $i=1,\ldots,k$ a decomposition $V_i=\bigoplus_{j=1}^{t_i}U_i^j$ with the properties described in Lemma~\ref{le:decomp}.

We may now define the sets $\mathcal{W}_j \subset \prod_{i=1}^k \mathrm{Gr}_{\ell_i}(V_i)$ mentioned in the statement of the theorem. By relabelling the indices $i$ if necessary we assume without loss of generality that $\max_i t_i = t_1$. Define $p\coloneqq \prod_{i=2}^k t_i = (\max_i t_i)^{-1} \prod_{i=1}^k t_i$. Define
\[\mathfrak{J}\coloneqq \{\mathfrak{j}=(j_1,\ldots,j_k)\in \mathbb{N}^k \colon j_1=1\text{ and }1 \leq j_i \leq t_i\text{ for all }i=2,\ldots,k\}\]
and for each $\mathfrak{j}=(j_i)_{i=1}^k \in \mathfrak{J}$ define
\[\mathcal{W}_{\mathfrak{j}}\coloneqq \left\{(\rho_i(g)U_i^{j_i})_{i=1}^k \colon g \in G\right\}\subset \prod_{i=1}^k \mathrm{Gr}_{\ell_i}(V_i)\]
and
\[\phi_{\mathfrak{j}}(g)\coloneqq \max_{(W_i)_{i=1}^k\in \mathcal{W}_{\mathfrak{j}}} \prod_{i=1}^k \|\rho_i(g)|_{W_i}\|^{\beta_i}\]
for every $g \in G$.
Note that $\#\mathfrak{J}=p$ and therefore there are exactly $p$ sets $\mathcal{W}_{\mathfrak{j}}$ and functions $\phi_{\mathfrak{j}}$ as required by the statement of the theorem.
We observe that if $(W_i)_{i=1}^k \in \mathcal{W}_{\mathfrak{j}}$ then $(\rho_i(g)W_i)_{i=1}^k \in \mathcal{W}_{\mathfrak{j}}$ for every $g \in G$ and hence in particular $(A^{(i)}_jW_i)_{i=1}^k \in \mathcal{W}_{\mathfrak{j}}$ for each $j=1,\ldots,N$. It follows that each $\mathcal{W}_{\mathfrak{j}}$ has the invariance property claimed in the statement of Theorem~\ref{th:irr}.
We observe that for each $\mathfrak{j}\in \mathfrak{J}$, $\phi_{\mathfrak{j}}$ and $\mathcal{W}_{\mathfrak{j}}$ meet the hypotheses of Theorem~\ref{pr:quam}, since it follows from Lemma~\ref{le:fix} that $(U_i^{j_i})_{i=1}^k$ is a $k$-tuple of $\ell_i$-dimensional subspaces which are fixed by $\rho_i(G^\circ)$, and since by Lemma \ref{le:ell-i-minimal} there can for each $i$ be no nonzero subspace of $V_i$ which is fixed by $\rho_i(G^\circ)$ but has dimension strictly less than $\ell_i$. Hence each $\mathcal{W}_{\mathfrak{j}}$ is a finite set and there exist a finite set $H' \subset H$ and a constant $\delta>0$ such that for each $\mathfrak{j} \in \mathfrak{J}$ we have
\begin{equation}\label{eq:del}\delta\phi_{\mathfrak{j}}(g_1)\phi_{\mathfrak{j}}(g_2) \leq \max_{h \in H'} \phi_{\mathfrak{j}}(g_1hg_2), \qquad  \phi_{\mathfrak{j}}(g_1g_2) \leq \phi_{\mathfrak{j}}(g_1)\phi_{\mathfrak{j}}(g_2) \end{equation}
for every $g_1,g_2 \in G$.

Define $\phi(g):=\prod_{i=1}^k \|\rho_i(g)\|^{\beta_i}$ for every $g \in G$. We claim that there is a constant $\tau>0$ such that 
\[\tau \phi(g) \leq \max_{\mathfrak{j} \in \mathfrak{J}} \phi_{\mathfrak{j}} (g) \leq \phi(g)\]
for every $g \in G$. The latter of the two inequalities is trivial, so we consider the former.

We first note that there is a constant $\kappa_1>0$ such that for any nonzero linear map $B_1 \colon V_1 \to V_1$ one has
\[\kappa_1 \|B_1\|^{\beta_1} \leq \max_{g \in G}\left\|B_1|_{\rho_1(g)U_1}\right\|^{\beta_1}.\]
(This maximum is well-defined since the set $\{\rho_1(g)U_1 \colon g \in G\}$ is finite by Lemma~\ref{le:ell-i-minimal}.)  To see this it is sufficient to consider the case $\|B_1\|=1$. If the result is false then by compactness there exists $B_1$ such that $\|B_1\|=1$ and $B_1|_{\rho_1(g)U_1}=0$ for all $g \in G$, and in particular $B_1|_{U_1^j}=0$ for all $j=1,\ldots,t_1$; but since $V_1=\bigoplus_{j=1}^{t_1}U_1^j$ we would have $B_1=0$ by linearity, a contradiction. We deduce the existence of the constant $\kappa_1>0$. In a closely-related fashion we assert that for each $i=2,\ldots,k$ there exists $\kappa_i>0$ such that for any nonzero linear map $B_i \colon V_i \to V_i$ and any $g \in G$ one has
\[\kappa_i \left\|B_i\right\|^{\beta_i} \leq \max_{1 \leq j \leq t_i} \left\|B_i|_{\rho_i(g)U_i^j}\right\|^{\beta_i}.\]
Clearly it suffices to prove this assertion individually for each $i$. By Lemma~\ref{le:ell-i-minimal} the vector space $\rho_i(g)U_i^j$ takes only finitely many values as $g$ varies over $G$, so it is also sufficient to prove the assertion for fixed $g \in G$. To prove the assertion we fix $i$ and $g$, reduce once more to the case $\|B_i\|=1$ and, applying compactness, note that if the result is false then we can find $B_i \colon V_i \to V_i$ with norm $1$ which is zero on $\bigoplus_{j=1}^{t_i} \rho_i(g)U_i^j = \rho_i(g) \left(\bigoplus_{j=1}^{t_i}U_i^j \right)= \rho_i(g)V_i=V_i$. This is clearly impossible and the existence of $\kappa_2,\ldots,\kappa_k>0$ follows.

Now fix $g \in G$ and observe that we may choose $g_0\in G$ such that 
\[\kappa_1\|\rho_1(g)\|^{\beta_1}\leq  \|\rho_1(g)|_{\rho_1(g_0)U_1}\|^{\beta_1}.\]
We then have 
\[\kappa_i \left\|\rho_i(g)\right\|^{\beta_i} \leq \max_{1 \leq j_i \leq t_i} \left\|\rho_i(g)|_{\rho_i(g_0)U_i^{j_i}}\right\|^{\beta_i}\]
for each $i=2,\ldots,k$ and therefore there exists $\mathfrak{j}=(1,j_2,\ldots,j_k)\in \mathfrak{J}$ such that
\[\left(\prod_{i=1}^k \kappa_i\right)\phi(g)=\prod_{i=1}^k \kappa_i\left\|\rho_i(g)\right\|^{\beta_i }\leq \prod_{i=1}^k \left\|\rho_i(g)|_{\rho_i(g_0)U_i^{j_i}}\right\|^{\beta_i} \leq\phi_{\mathfrak{j}}(g) \]
by the definition of $\phi_{\mathfrak{j}}(g)$. Since $g \in G$ was arbitrary we conclude that 
\[\tau \phi(g) \leq \max_{\mathfrak{j} \in \mathfrak{J}} \phi_{\mathfrak{j}} (g) \leq \phi(g)\]
for all $g \in G$ as required, where $\tau\coloneqq \prod_{i=1}^k \kappa_i>0$. 

We are finally ready to investigate the equilibrium states of the potential
\[\Phi(\iii)\coloneqq \phi(A_{\iii})=\prod_{i=1}^k \left\|A_{\iii}^{(i)}\right\|^{\beta_i}.\]
For each $\mathfrak{j}\in \mathfrak{J}$ we define a potential $\Phi_{\mathfrak{j}}\colon \Sigma_N^* \to (0,+\infty)$ by
\[\Phi_{\mathfrak{j}}(\iii)\coloneqq \phi_{\mathfrak{j}}(A_{\iii})=\max_{(W_i)_{i=1}^k\in\mathcal{W}_{\mathfrak{j}}}\prod_{i=1}^k \left\|A_{\iii}^{(i)}|_{W_i}\right\|^{\beta_i}\]
and note that 
\begin{equation}\label{eq:max}\tau \Phi(\iii) \leq\max_{\mathfrak{j}\in \mathfrak{J}}\Phi_{\mathfrak{j}}(\iii) \leq \Phi(\iii)\end{equation}
for every $\iii \in \Sigma_N^*$ as required in the statement of Theorem~\ref{th:irr}. It follows directly that $P(\Phi)\geq P(\Phi_{\mathfrak{j}})$ for every $\mathfrak{j}\in \mathfrak{J}$. Since $H'$ is finite we may choose a finite set of words $F$ such that  $H'=\{A_{\iii} \colon \iii \in F\}$. For fixed $\mathfrak{j}\in \mathfrak{J}$ we may for every $\iii,\jjj$ find a corresponding word $\kkk \in F$ such that
\[\Phi_{\mathfrak{j}}(\iii\kkk\jjj) \geq \delta \Phi_{\mathfrak{j}}(\iii)\Phi_{\mathfrak{j}}(\jjj)\] 
by virtue of \eqref{eq:del}, so each $\Phi_{\mathfrak{j}}$ is quasimultiplicative. Clearly each $\Phi_{\mathfrak{j}}$ is also submultiplicative. It follows by Proposition~\ref{pr:kare} that for each $\mathfrak{j}\in \mathfrak{J}$ the potential $\Phi_{\mathfrak{j}}$ has a unique equilibrium state $\mu_{\mathfrak{j}}$ and this measure satisfies a Gibbs inequality with respect to $\Phi_{\mathfrak{j}}$.

Suppose now that $\mu$ is an ergodic equilibrium state of $\Phi$. By the subadditive ergodic theorem we have for $\mu$-a.e. $x \in \Sigma_N$
\[\Lambda(\Phi,\mu)=\lim_{n\to\infty} \frac{1}{n}\log \Phi(x|_n)\]
and
\[\Lambda(\Phi_{\mathfrak{j}},\mu)=\lim_{n\to\infty} \frac{1}{n}\log \Phi_{\mathfrak{j}} (x|_n)\]
for every $\mathfrak{j}\in \mathfrak{J}$. Using \eqref{eq:max} it follows that for $\mu$-a.e. $x\in \Sigma_N$ we have
\begin{align*}\Lambda(\Phi,\mu)&=\lim_{n\to\infty} \frac{1}{n}\log \Phi(x|_n)\\
&=\lim_{n\to\infty} \frac{1}{n}\log \max_{\mathfrak{j}\in \mathfrak{J}}\Phi_{\mathfrak{j}}(x|_n)\\
&= \max_{\mathfrak{j}\in \mathfrak{J}}\lim_{n\to\infty}\frac{1}{n}\log \Phi_{\mathfrak{j}}(x|_n)=\max_{\mathfrak{j}\in \mathfrak{J}}\Lambda(\Phi_{\mathfrak{j}},\mu).\end{align*}
Hence there exists $\mathfrak{j}\in \mathfrak{J}$ such that $\Lambda(\Phi,\mu) = \Lambda(\Phi_{\mathfrak{j}},\mu)$ and therefore
\[P(\Phi)=h(\mu)+\Lambda(\Phi,\mu)=h(\mu)+\Lambda(\Phi_{\mathfrak{j}},\mu) \leq P(\Phi_{\mathfrak{j}}) \leq P(\Phi)\]
so that $h(\mu)+\Lambda(\Phi_{\mathfrak{j}},\mu) = P(\Phi_{\mathfrak{j}})$. We conclude that $\mu$ is an equilibrium state of $\Phi_{\mathfrak{j}}$ and is therefore equal to $\mu_{\mathfrak{j}}$. The proof of the theorem is complete.

\section{The reducible case: proof of Theorem~\ref{th:force-diagonal}}\label{se:six}

Let us first prove the existence of representations of the linear maps $A_j^{(i)}$ as matrices of the form \eqref{eq:mat}. Clearly we may prove this statement separately for each individual $i$. Fix $i$ and let $n_i$ be the maximum integer such that there exists a basis for $V$ in which we may write
\begin{equation}\label{eq:tri2}A^{(i)}_j =
    \begin{pmatrix}
      A_j^{(i,1)} & * & \cdots & * \\
      0           & A_j^{(i,2)} & \cdots & * \\
      \vdots      & \vdots      & \ddots & \vdots \\
      0           & 0           & \cdots & A_j^{(i,n_i)}
    \end{pmatrix}
\end{equation}
for every $j=1,\ldots,N$ where each $A_j^{(i,r)}$ is an invertible square matrix with dimension $d_{i,r}$. Obviously a largest such integer exists since $1$ is such an integer and since no such integer can be greater than $d_i$. We observe that for each $r$ the tuple $(A_1^{(i,r)},\ldots,A_N^{(i,r)})$ must be irreducible: if it admits an invariant subspace $W$ with $0<\dim W <d_{i,r}$ then up to a suitable simultaneous change of basis we have
\[A_j^{(i,r)} = \begin{pmatrix}A_j^{(i,r)}|_W & * \\ 0 & *\end{pmatrix}\]
which implies that the block upper triangularisation \eqref{eq:tri2} may be refined so as to have $n_i+1$ diagonal blocks instead of $n_i$, contradicting the maximality of $n_i$. This completes the proof of the existence of a triangularisation \eqref{eq:mat} with irreducible diagonal blocks as claimed in the statement of Theorem~\ref{th:force-diagonal}. By relabelling the indices $i=1,\ldots,k$ if necessary, for the remainder of the proof we shall assume that $\min_{1 \leq i \leq k} n_i/d_i=n_1/d_1$. Our desired bound for the number of ergodic equilibrium states of $\Phi$ is therefore $n_1 \prod_{i=2}^k d_i$.

Let $\mathfrak{R}: =\{\mathfrak{r} =(r_1,\ldots,r_k)\colon 1 \leq r_i \leq n_i\text{ for all }i=1,\ldots,k\}$ and let us define potentials $\Phi, \Phi_{\mathfrak{r}}\colon \Sigma_N^* \to (0,+\infty)$ by 
\[\Phi(\iii)=\prod_{i=1}^k \left\|A_\iii^{(i)}\right\|^{\beta_i},\qquad \Phi_{\mathfrak{r}}(\iii)=\prod_{i=1}^k \left\|A_\iii^{(i,r_i)}\right\|^{\beta_i}\]
where $\mathfrak{r} \in \mathfrak{R}$. For each $\mu \in \mathcal{M}_\sigma$, $\mathfrak{r}\in\mathfrak{R}$  and $i \in \{1,\ldots,k\}$ let us define
\[\Lambda\left(A^{(i)},\mu\right) \coloneqq \lim_{n \to \infty} \frac{1}{n}\sum_{|\iii|=n}^N \mu([\iii])\log \left\|A_\iii^{(i)}\right\|\]
and
\[\Lambda\left(A^{(i,r_i)},\mu\right) \coloneqq \lim_{n \to \infty} \frac{1}{n}\sum_{|\iii|=n}^N \mu([\iii])\log \left\|A_\iii^{(i,r_i)}\right\|\]
and note that
\[\Lambda\left(\Phi,\mu\right)=\sum_{i=1}^k \beta_i \Lambda\left(A^{(i)},\mu\right)\]
by direct calculation. If $\mu$ is ergodic then by standard arguments (see e.g.\ \cite[p.~129--130]{LArnold}) the equation \eqref{eq:tri2} implies
\[\Lambda\left(A^{(i)},\mu\right)=\max_{1 \leq r_i \leq n_i} \Lambda\left(A^{(i,r_i)},\mu\right).\]
It follows that if $\mu$ is ergodic then 
\begin{align*}\Lambda(\Phi,\mu)&=\sum_{i=1}^N \beta_i \max_{1 \leq r_i \leq n_i}\Lambda\left(A^{(i,r_i)},\mu\right)\\&=\max_{\mathfrak{r} \in \mathfrak{R}} \sum_{i=1}^N \beta_i\Lambda\left(A^{(i,r_i)},\mu\right)=\max_{\mathfrak{r}\in\mathfrak{R}} \Lambda\left(\Phi_{\mathfrak{r}},\mu\right).\end{align*}
Hence by the subadditive variational principle \eqref{eq:svp} we have $P(\Phi)=\max_{\mathfrak{r} \in \mathfrak{R}}P(\Phi_{\mathfrak{r}})$ and if $\mu$ is an ergodic equilibrium state of $\Phi$ then it is necessarily an equilibrium state of $\Phi_{\mathfrak{r}}$ for some $\mathfrak{r}\in \mathfrak{R}$ as claimed in the statement of Theorem~\ref{th:force-diagonal}. By Theorem~\ref{th:irr}, for each $\mathfrak{r}=(r_1,\ldots,r_k) \in \mathfrak{R}$ there exist integers $t_{1,r_1},\ldots,t_{k,r_k}$ with $t_{i,r_i} \divides d_{i,r_i}$ for each $i$ such that the number of ergodic equilibrium states of the potential $\Phi_{\mathfrak{r}}$ is bounded above by
\[\left(\max_{1 \leq i \leq k} t_{i,r_i}\right)^{-1}\left(\prod_{i=1}^k t_{i,r_i} \right) \leq \prod_{i=2}^k t_{i,r_i} \leq \prod_{i=2}^k d_{i,r_i} \]
and all of the equilibrium states of $\Phi_{\mathfrak{r}}$  are fully supported. It follows that every ergodic equilibrium state of $\Phi$ is fully supported and the number of ergodic equilibrium states of $\Phi$ is bounded above by $\sum_{\mathfrak{r}\in \mathfrak{R}} \prod_{i=2}^k d_{i,r_i}$. Let us write $\hat{d}_{i,r_i}=1$ when $i=1$ and $\hat{d}_{i,r_i}=d_{i,r_i}$ otherwise; then the number of ergodic equilibrium states of $\Phi$ is bounded by 
\[\sum_{\mathfrak{r}\in \mathfrak{R}} \prod_{i=2}^k d_{i,r_i}= \sum_{\mathfrak{r}\in \mathfrak{R}} \prod_{i=1}^k \hat{d}_{i,r_i} =\prod_{i=1}^k \sum_{r_i=1}^{n_i} \hat{d}_{i,r_i} = n_1\prod_{i=2}^k d_i= \left(\min_{1 \leq i \leq k} \frac{n_i}{d_i}\right)\prod_{i=1}^k d_i\]
as required. Since every equilibrium state of $\Phi$ is a linear combination of ergodic equilibrium states, every equilibrium state of $\Phi$ is fully supported. Clearly $n_i \leq d_i$ for every $i$ by the definition of $n_i$, so $n_i/d_i \leq 1$ for every $i$. It follows that if the number of ergodic equilibrium states is exactly $\prod_{i=1}^k d_i$ then necessarily $n_i/d_i=1$ for every $i$, which is to say $n_i\equiv d_i$ and thus for every $i=1,\ldots,k$ the tuple $(A_1^{(i)},\ldots,A_N^{(i)})$ is simultaneously triangularisable by the definition of $n_i$. The proof of the theorem is complete.

\section{Sharp bounds for the number of ergodic equilibrium states}\label{se:seven}

It is not clear to what extent the upper bound $p\coloneqq (\max_i t_i)^{-1}\prod_{i=1}^k t_i$ for the number of ergodic equilibrium states in Theorem~\ref{th:irr} is optimal. Let us briefly sketch an example which shows that $2^{k-1}$ ergodic equilibrium states can exist in the case where $t_i \equiv 2$. For simplicity only the case $k=2$ will be examined in detail since for larger $k$ the notation quickly becomes cumbersome. We will find it convenient to index our matrices $A_j^{(i)}$ starting from $j=0$ instead of from $j=1$ while otherwise retaining our previous notation for $\Sigma_N$, $\Sigma_N^*$ and so forth. For related reasons we let $e_0,e_1$ denote the standard basis for $\mathbb{R}^2$. Define 
\[A_0^{(1)}=\begin{pmatrix}
  0&2\\1&0
      \end{pmatrix},\quad
      A_1^{(1)}=\begin{pmatrix}
  0&1\\2&0
      \end{pmatrix},\quad
      A_2^{(1)}=\begin{pmatrix}
  0&2\\1&0
      \end{pmatrix},\quad
      A_3^{(1)}=\begin{pmatrix}
  0&1\\2&0
      \end{pmatrix},
      \]
\[A_0^{(2)}=\begin{pmatrix}
  0&2\\1&0
      \end{pmatrix},\quad
      A_1^{(2)}=\begin{pmatrix}
  0&2\\1&0
      \end{pmatrix},\quad
      A_2^{(2)}=\begin{pmatrix}
  0&1\\2&0
      \end{pmatrix},\quad
      A_3^{(2)}=\begin{pmatrix}
  0&1\\2&0
      \end{pmatrix}
      \]
and define a potential $\Phi \colon \Sigma_4^* \to (0,+\infty)$ by
\[\Phi(\iii)=\prod_{i=1}^2 \left\|A_\iii^{(i)}\right\| = \left\|A_\iii^{(1)} \otimes A_\iii^{(2)}\right\|\]
where $\otimes$ denotes the tensor product. We note that the hypotheses of Theorem~\ref{th:irr} are satisfied. Let $B_\iii\coloneqq A_\iii^{(1)}\otimes A_\iii^{(2)}$ for each $\iii \in \Sigma_4^*$. In the basis $e_0 \otimes e_0$, $e_1 \otimes e_1$, $e_0 \otimes e_1$, $e_1 \otimes e_0$ for $\mathbb{R}^2 \otimes \mathbb{R}^2$ we have
\[B_0=\begin{pmatrix}
0&4&0&0\\
1&0&0&0\\
0&0&0&2\\
0&0&2&0
\end{pmatrix},\qquad
B_1=\begin{pmatrix}
0&2&0&0\\
2&0&0&0\\
0&0&0&1\\
0&0&4&0
\end{pmatrix},
\]
\[B_2=\begin{pmatrix}
0&2&0&0\\
2&0&0&0\\
0&0&0&4\\
0&0&1&0
\end{pmatrix},\qquad
B_3=\begin{pmatrix}
0&1&0&0\\
4&0&0&0\\
0&0&0&2\\
0&0&2&0
\end{pmatrix}.
\]
It follows that if we define
\[C_0=\begin{pmatrix}
  0&4\\1&0
      \end{pmatrix},\quad
      C_1=\begin{pmatrix}
  0&2\\2&0
      \end{pmatrix},\quad
      C_2=\begin{pmatrix}
  0&2\\2&0
      \end{pmatrix},\quad
      C_3=\begin{pmatrix}
  0&1\\4&0
      \end{pmatrix},
      \]
      \[D_0=\begin{pmatrix}
  0&2\\2&0
      \end{pmatrix},\quad
      D_1=\begin{pmatrix}
  0&1\\4&0
      \end{pmatrix},\quad
      D_2=\begin{pmatrix}
  0&4\\1&0
      \end{pmatrix},\quad
      D_3=\begin{pmatrix}
  0&2\\2&0
      \end{pmatrix}
      \]
and
\[\Phi_C(\iii)=\|C_\iii\|,\qquad \Phi_D(\iii)=\|D_\iii\|\]
for every $\iii \in \Sigma_4$ then $\Phi =\max\{\Phi_C,\Phi_D\}$. Clearly $P(\Phi_C)=P(\Phi_D)$ and it follows that $P(\Phi)=P(\Phi_C)=P(\Phi_D)$. We may easily deduce that an ergodic measure $\mu$ on $\Sigma_4$ is an equilibrium state of $\Phi$ if and only if it is an equilibrium state of either $\Phi_C$ or $\Phi_D$. Clearly $(C_0,\ldots,C_3)$ and $(D_0,\ldots,D_3)$ are both irreducible, so by Corollary~\ref{co:ral} there exist measures $\mu_C$, $\mu_D$ on $\Sigma_4$ and a constant $K>0$ such that
\[K^{-1}\|C_\iii\|\leq \frac{\mu_C([\iii])}{e^{-|\iii|P(\Phi)}} \leq K \|C_\iii\|,\qquad K^{-1}\|D_\iii\|\leq \frac{\mu_D([\iii])}{e^{-|\iii|P(\Phi)}} \leq K \|D_\iii\|\]
for all $\iii \in \Sigma_4^*$, where we have used the fact that $P(\Phi)=P(\Phi_C)=P(\Phi_D)$. If the measures $\mu_C$, $\mu_D$ were identical then these two Gibbs inequalities together would imply
\[\lim_{n \to \infty} \left\|\left(C_0C_3\right)^n\right\|^{\frac{1}{n}}=\lim_{n \to \infty} \left\|\left(D_0D_3\right)^n\right\|^{\frac{1}{n}} \]
but the former limit is $16$ and the latter is $4$, so the two measures are distinct and $\Phi$ has two ergodic equilibrium measures which is the maximum permitted by Theorem~\ref{th:irr}.

More generally, given $k \geq 1$ we may proceed as follows. Let $\mathfrak{b}_i(j)$ denote the $i^{\mathrm{th}}$ binary digit of the integer $j \in \{0,\ldots,2^{k}-1\}$ starting from the least significant digit so that $j=\sum_{i=1}^k \mathfrak{b}_i(j)2^{i-1}$. Define for each $j=0,\ldots,2^k-1$ and $i=1,\ldots,k$
\[A^{(i)}_j=\left\{\begin{array}{cl}
  \begin{pmatrix}
  0&2\\1&0
      \end{pmatrix}&\text{if }\mathfrak{b}_i(j)=0,\\
        \begin{pmatrix}
  0&1\\2&0
      \end{pmatrix}&\text{if }\mathfrak{b}_i(j)=1\end{array}
\right.\]
and define $\Phi(\iii)=\prod_{i=1}^k\|A_\iii^{(i)}\|=\|\otimes_{i=1}^k A_\iii^{(i)}\|=\|B_\iii\|$, say, where $B_\iii \in GL(\otimes_{i=1}^k \mathbb{R}^2)$. In a suitable basis for $\otimes_{i=1}^k\mathbb{R}^2$ we may write each $B_j$ as a direct sum of $2^{k-1}$ matrices of dimension $2\times 2$, decomposing $(B_0,\ldots,B_{2^k-1})$ into $2^{k-1}$ irreducible $2^k$-tuples each of which has the same pressure and contributes a distinct equilibrium state. Each 2-dimensional invariant subspace which corresponds to a $2 \times 2$ block is spanned by a pair of vectors of the form $\otimes_{i=1}^k e_{\mathfrak{b}_i(j)}, \otimes_{i=1}^k e_{\mathfrak{b}_i(2^k-1-j)}$ for some $j \in \{0,\ldots,2^{k-1}-1\}$. When $j$ is fixed the limit
\[\lim_{n \to \infty} \left\|\left(B_jB_{2^k-j}\right)^n\Big|_{\mathrm{span}\left\{\otimes_{i=1}^k e_{\mathfrak{b}_i(\ell)},\otimes_{i=1}^k e_{\mathfrak{b}_i(2^k-1-\ell)}\right\}}\right\|^{\frac{1}{n}}\]
is maximised only when $\ell \in \{j,2^k-1-j\}$ which implies that each of these $2^{k-1}$ invariant subspaces contributes a distinct equilibrium state by analogous reasoning to the case $k=2$.  We leave further details to the reader. We also leave to the reader the problem of showing that the maximum number of ergodic equilibrium states in Theorem~\ref{th:force-diagonal} can be attained in suitable cases where all of the matrices are diagonal, by adapting the argument of \cite[Proposition 5.3]{KaMo16}.

The preceding example shows that the upper bound for the number of ergodic equilibrium states in Theorem~\ref{th:irr} is sharp in at least some cases, but it appears to be more difficult to construct examples with large numbers of ergodic equilibrium states when the integers $t_i$ are allowed to vary with $i$. We pose the following question:
\begin{question}
Does Theorem~\ref{th:irr} remain true if $p\coloneqq (\max_i t_i)^{-1}\prod_{i=1}^k t_i$ is replaced with $p'\coloneqq (\mathrm{lcm}(t_1,\ldots,t_k))^{-1}\prod_{i=1}^k t_i$? More generally, what is the smallest value of $p=p(t_1,\ldots,t_k)$ for which the statement of Theorem~\ref{th:irr} remains valid?
\end{question}

As was remarked in the introduction the bound ${d\choose \lfloor s \rfloor}{d \choose \lceil s \rceil}$ for the number of ergodic equilibrium states of Falconer's singular value function $\varphi^s$ is known not to be optimal in dimension three. We advance the following conjecture on the number of ergodic equilibrium states:
\begin{conjecture}\label{co:njecture}
Let $(A_1,\ldots,A_N)\in GL_d(\mathbb{R})$ and $s \in (0,d)\setminus \mathbb{Z}$. Then the maximum possible number of ergodic equilibrium states of $\varphi^s$ is precisely $(d-\lfloor s \rfloor){d \choose \lfloor s \rfloor}=\lceil s \rceil {d \choose \lceil s \rceil}$.
\end{conjecture}
Theorem \ref{th:intro2} resolves this conjecture positively when $0<s<1$, when $d-1<s<d$ or when $s \in (0,d)\cap\mathbb{Z}$. If $s \in (1,d-1)\setminus \mathbb{Z}$ and $(A_1^{\wedge \lfloor s \rfloor},\ldots,A_N^{\wedge \lfloor s \rfloor})$ is irreducible then the number of ergodic equilibrium states is bounded by ${d \choose \lceil s \rceil}$ as a consequence of Theorem \ref{th:force-diagonal}, and similarly if $(A_1^{\wedge \lceil s \rceil},\ldots,A_N^{\wedge \lceil s \rceil})$ is irreducible then it is bounded by ${d \choose \lfloor s \rfloor}$. The unresolved cases therefore occur when $s \in (1,d-1)\setminus \mathbb{Z}$ and both $(A_1^{\wedge \lfloor s \rfloor},\ldots,A_N^{\wedge \lfloor s \rfloor})$ and $(A_1^{\wedge \lceil s \rceil},\ldots,A_N^{\wedge \lceil s \rceil})$ are reducible. We also ask:
\begin{question}\label{qu:maxbound}
Let $(A_1,\ldots,A_N)\in GL_d(\mathbb{R})$ and $s \in (0,d)\setminus \mathbb{Z}$. If $(A_1,\ldots,A_N)$ has at least $(d-\lfloor s \rfloor){d \choose \lfloor s \rfloor}$ ergodic equilibrium states with respect to $\varphi^s$, does it follow that $(A_1,\ldots,A_N)$ is simultaneously triangularisable?
\end{question}
Since it was shown in \cite[Theorem~D]{KaMo16} that the maximum possible number of ergodic equilibrium states is $(d-\lfloor s \rfloor){d \choose \lfloor s \rfloor}$ when the matrices $A_1,\ldots,A_N$ are simultaneously triangularisable, a positive answer to Question~\ref{qu:maxbound} would resolve Conjecture \ref{co:njecture}. A positive answer to Question~\ref{qu:maxbound} in the case $d=2$ follows directly from Theorem \ref{th:force-diagonal}, and we are able to give a positive answer for $d=3$ and $d=4$ by a somewhat laborious case-by-case analysis. In dimensions higher than four the question remains open.

We remark that the upper bound for the number of ergodic equilibrium states given by Theorem \ref{th:intro2} was derived by disregarding the relationship between the tuples $(A_1^{\wedge \lfloor s \rfloor},\ldots,A_N^{\wedge \lfloor s \rfloor})$ and $(A_1^{\wedge \lceil s \rceil},\ldots,A_N^{\wedge \lceil s \rceil})$ and instead treating the two tuples as if they were entirely unrelated. In order to obtain sharp upper bounds for the number of ergodic equilibrium states it seems intuitively reasonable that the relationship between the two tuples should be exploited in some way, perhaps by the use of partial flags of $\lfloor s \rfloor$- and $\lceil s \rceil$-dimensional subspaces of $\mathbb{R}^d$ instead of pairings between subspaces of $\WEDGE^{\lfloor s \rfloor}\mathbb{R}^d$ and $\WEDGE^{\lceil s \rceil}\mathbb{R}^d$.

\section{Extensions of Corollary \ref{co:xsapples}}\label{se:mon}

Let us say that an affine iterated function system $(T_1,\ldots,T_N)$ has the \emph{strict monotonicity property for Hausdorff dimension} if for every affine IFS $(T_1',\ldots,T_{N-1}')$ formed by deleting one of the contractions $T_i$ and retaining the rest, the attractor of $(T_1',\ldots,T_{N-1}')$ has strictly smaller Hausdorff dimension than the attractor of $(T_1,\ldots,T_N)$. (Here and throughout this section we of course assume $N \geq 2$.) Corollary \ref{co:xsapples} demonstrates that if the contractions $T_i$ are invertible and the affinity dimension of $(T_1,\ldots,T_N)$ equals the Hausdorff dimension of its attractor, then $(T_1,\ldots,T_N)$ has the strict monotonicity property for the Hausdorff dimension. However this condition is not necessary for the strict monotonicity property to hold, since it is easy to see that the strict monotonicity property is also satisfied for the Bedford-McMullen carpets studied in \cite{Be84,Mc84} whose Hausdorff dimension is smaller than their affinity dimension. It is therefore natural to ask for more general conditions under which the strict monotonicity property for the Hausdorff dimension holds.

We recall that an iterated function system $(T_1,\ldots,T_N)$ acting on $\mathbb{R}^d$ is said to satisfy the Open Set Condition if there exists a nonempty open $U \subset \mathbb{R}^d$ such that the images $T_1U,\ldots,T_NU$ are pairwise disjoint subsets of $U$. We say that $(T_1,\ldots,T_N)$ satisfies the \emph{Strong Open Set Condition} if additionally $U$ intersects the attractor of $(T_1,\ldots,T_N)$, or equivalently if there exists $\mathtt{i} \in \Sigma_N^*$ such that $\overline{T_{\mathtt{i}}U}\subset U$. The Open Set Condition cannot be sufficient for the strict monotonicity property for the Hausdorff dimension to hold, since an example of G. A. Edgar \cite[Example 1]{Ed92} shows that an affine iterated function system defined by invertible affinities can satisfy the Open Set Condition but have a singleton set as its attractor. On the other hand the attractor cannot be a singleton set when the Strong Open Set Condition holds. We note the following question which has been attributed to J. Schmeling:
\begin{question}
Let $(T_1,\ldots,T_N)$ be an affine iterated function system acting on $\mathbb{R}^d$ which satisfies the Strong Open Set Condition and such that every $T_i$ is invertible. Does $(T_1,\ldots,T_N)$ satisfy the strict monotonicity property for the Hausdorff dimension?
\end{question}
One may of course also define and investigate the strict monotonicity property for other notions of dimension, such as the box dimension and packing dimension.


\bigskip
\begin{ack}
The authors would like to thank De-Jun Feng, Rafael Potrie, and Pablo Shmerkin for helpful discussions and suggestions. The authors also thank the referee for their careful reading of the paper.
\end{ack}

\bibliographystyle{acm}
\bibliography{falconer-biblio}
\end{document}